\documentclass[11pt]{articlefederico}
\usepackage{graphicx}
\usepackage{amsfonts }
\usepackage{amsmath}
\usepackage{fullpage}
\usepackage{amssymb}
\usepackage{amsthm}
\usepackage{tikz}

\theoremstyle{definition}
\newtheorem{theorem}{Theorem}[section]
\newtheorem{teo}{Theorem}[section]
\newtheorem{cor}[theorem]{Corollary}
\newtheorem{lem}[theorem]{Lemma}

\newtheorem{conj}[theorem]{Conjecture}
\newtheorem{prop}[theorem]{Proposition}
\newtheorem{defi}[theorem]{Definition}

\newtheorem{ex}[theorem]{Example}

\makeatletter
\newtheorem*{rep@theorem}{\rep@title}
\newcommand{\newreptheorem}[2]{%
\newenvironment{rep#1}[1]{%
 \def\rep@title{#2 \ref{##1}}%
 \begin{rep@theorem}}%
 {\end{rep@theorem}}}
\makeatother

\newreptheorem{theorem}{Theorem}

\newcommand{\I}{{\cal I}}
\newcommand{\B}{{\cal B}}

\newcommand{\RR}{{\cal R}}
\newcommand{\Circ}{\mathrm{Circ}}
\newcommand{\Cocirc}{\textrm{Cocirc}}
\newcommand{\supp}{\textrm{supp}}

\begin{document}
\title{\textsf{The topology of the external activity complex of a matroid}}
\author{
\textsf{Federico Ardila\footnote{\noindent \textsf{San Francisco State University, San Francisco, USA; Universidad de Los Andes, Bogot\'a, Colombia. federico@sfsu.edu}}} \\
%\small Department of Mathematics\\[-0.8ex]
%\small San Fransico State University \\[-0.8ex]
%\small San Francisco, CA, USA\\ [-0.8.ex]
%\small \texttt{federico@sfsu.edu}
\and \textsf{Federico Castillo\footnote{\noindent \textsf{University of California, Davis, USA. fcastillo@math.ucdavis.edu}}} \\
%\small Department of Mathematics \\[-0.8ex]
%\small University of California at Davis\\[-0.8ex]
%\small Davis, CA, USA\\[-0.8ex]
%\small \texttt{fcastillo@math.ucdavis.edu}
\and \textsf{Jos\'e Alejandro Samper\footnote{\noindent \textsf{University of Washington, Seattle, USA. samper@math.washington.edu\newline Ardila was partially supported by the US National Science Foundation CAREER Award DMS-0956178 and the SFSU-Colombia Combinatorics Initiative.}}}} 
%\small Department of Mathematics\\[-0.8ex]
%\small University of Washington\\[-0.8ex]
%\small Seattle, WA 98195-4350, USA\\[-0.8ex]
%\small \texttt{samper@math.washington.edu}

\date{}

\maketitle

\begin{abstract}
We prove that the external activity complex $\textrm{Act}_<(M)$ of a matroid is shellable. In fact, we show that every linear extension of LasVergnas's external/internal order $<_{ext/int}$ on $M$ provides a shelling of $\textrm{Act}_<(M)$. We also show that every linear extension of LasVergnas's internal order $<_{int}$ on $M$ provides a shelling of the independence complex $IN(M)$. 
As a corollary, $\textrm{Act}_<(M)$ and $M$ have the same $h$-vector.
We prove that, after removing its cone points, the external activity complex is contractible if $M$ contains $U_{3,1}$ as a minor, and a sphere otherwise.

\end{abstract}

\section{\textsf{Introduction}}

\textbf{\textsf{Wider context of this work.}}
Matroid theory is a combinatorial theory of independence which has its roots in linear algebra and graph theory, but which turns out to have deep connections with many fields. There are natural notions of independence in linear algebra, graph theory, matching theory, the theory of field extensions, and the theory of routings, among others. Matroids capture the combinatorial essence that those notions share.

A matroid can be described in many equivalent ways, arising from the many contexts where matroids are found: the bases, the circuits, the lattice of flats, and the matroid polytope, among others. One important approach, which is the most relevant one to this paper, has been to model a matroid in terms of a simplicial or polyhedral complex. In fact, most of these topological models arise naturally in algebraic and geometric contexts, and offer new tools to prove combinatorial theorems. A celebrated recent example is the proof by Huh \cite{Huh} and Huh and Katz \cite{HuhKatz} of Rota's 1971 conjecture \cite{Rota71} that the coefficients of the characteristic polynomial of a linear matroid are unimodal. A key ingredient of this proof is the Bergman complex $\mathcal{B}(M)$ described below.

Let us describe a few constructions of this flavor, and provide a few references for the interested reader. The notion of shellability is a very useful unifying tool in this approach, as explained in \cite{bjorner}.
%
%
%, postponing the definitions to Section \ref{sec:background}. We select a few references from the extensive bibliography on each of these objects; \cite{bjorner} is particularly relevant.

 $\bullet$ \cite{Provan80, bjorner}
The \emph{independence complex} or \emph{matroid complex} $IN(M)$ is homotopy equivalent to a wedge of $T_M(0,1) = |\mu(M^*)|$ spheres of dimension $r(M)-1$ if $M$ is coloopless. This complex is shellable, and its shelling polynomial is $T_M(x,1)$. The shellability of $IN(M)$  naturally leads to the important notions of internal and external activity of $M$. 

 $\bullet$ \cite{Whitney32, Rota64, Brylawski77a}
The \emph{broken circuit complex} $\overline{BC}_<(M)$ is, a cone over a space homotopy equivalent to a wedge of $|\beta(M)|$ spheres of dimension $r(M)-2$. It can be naturally embedded into $IN(M)$. It is shellable and its shelling polynomial is $T_M(x,0)$. The embedding is a combinatorial witness of such a result. Its face numbers equal the coefficients of the characteristic polynomial of $M$ up to sign. 

 $\bullet$ \cite{Folkman66, Stanley77, OrlikSolomon80}
the (proper part of the) \emph{order complex of the lattice of flats} $\Delta(L_M \backslash \{\widehat{0}, \widehat{1}\})$ is homotopy equivalent to a wedge of $T_M(1,0) = |\mu(M)|$ spheres of dimension $r(M)-2$. It is shellable. %, and its shelling polynomial is \comment{Cual es?}\comment{Estuve buscando un buen rato. Bjorner calcula los restriction sets en terminos del EL shelling que viene del poset, pero no parece que de nada bonito. No se que opinan, pero yo creo que no hay mucho por decir. Hay una descripci—n de $h_i$ contando cierta cantidad de cadenas en un labelling 'natural' de los covers del poset, pero no parece ser nada bonito,}. 
This is a motivating example for the theory of Cohen Macaulay posets. It also arises naturally in Orlik and Solomon's presentation of the cohomology of the complement of a complex hyperplane arrangement.

 $\bullet$ \cite{Sturmfelsbook, ArdilaKlivans}
The \emph{Bergman complex} $\mathcal{B}(M)$ is the link of the origin in the \emph{tropical linear space} $\textrm{Trop}(M)$. It is not always simplicial. Though not obvious from its definition, $\mathcal{B}(M)$  is a coarsening of $\Delta(L_M \backslash \{\widehat{0}, \widehat{1}\})$ and hence shares its topological properties. These complexes are fundamental objects in tropical geometry because $\textrm{Trop}(M)$ is the tropical analog of a linear space.

\medskip

The purpose of this paper is to describe a new member of this family.

\medskip

 $\bullet$ 
The \emph{external activity complex} ${\textrm{Act}}_<(M)$ is, after removing cone points, either contractible or a sphere of dimension $n+r-1-|AE(M)|$
%$2|E| - |AE(M)| - 1$, 
where $AE(M)$ is the set of externally absolute elements. It contains a copy of $IN(M)$ as a subcomplex. It is shellable, and its shelling polynomial is $T_M(x,1)$.
Its shellability is closely related to Las Vergnas's \emph{active orders} on the bases of $M$. 

\medskip

Hence the external activity complex  sheds new light on the shelling polynomial $T_M(x,1)$ of a matroid $M$. This is a subject of great attention thanks to Stanley's 1977 \emph{$h$-vector conjecture}, one of the most intriguing open problems in matroid theory:

\begin{conj} \cite{Stanley77} For any matroid $M$, there exists a set $X$ of monomials such that:

- if $m$ and $m'$ are monomials such that $m \in X$ and $m'|m$, then $m' \in X$,

- all the maximal monomials in $X$ have the same degree,

- there are exactly $h_i$ monomials of degree $i$ in $X$, where $\sum_i h_ix^{r-i} = T(x,1)$.
\end{conj}

\noindent This conjecture has been proved, using rather different methods, for several families: cographic matroids,  \cite{Merinocographic}, lattice path matroids \cite{Schweig}, cotransversal matroids \cite{Ohcotransversal}, paving matroids \cite{Merinoetal}, and matroids up to rank $4$ or corank $2$ \cite{DeLoeraKemperKlee, KleeSamper}. The general case remains open.

\bigskip

\noindent
\textbf{\textsf{Motivation for this work.}}
The \emph{external activity complex} $\textrm{Act}_<(M)$ of a matroid is a simplicial complex associated to a matroid $M$ and a linear order $<$ on its ground set. This complex arose in work of the first author with Adam Boocher \cite{ArdilaBoocher}. They started with a linear subspace $L$ of affine space $\mathbb{A}^n$ with a chosen system of coordinates. There is a natural  embedding $\mathbb{A}^n \hookrightarrow (\mathbb{P}^1)^n$ into a product of projective lines, and they considered the closure $\widetilde{L}$ of $L$ in $(\mathbb{P}^1)^n$. They proved that many geometric and algebraic invariants of the variety $\widetilde{L}$ are determined by the matroid of $L$.

As is common in combinatorial commutative algebra, a key ingredient of \cite{ArdilaBoocher} was to consider the initial ideals $\textrm{in}_<\widetilde{L}$ under various term orders. These initial ideals are the Stanley-Reisner ideals of the external activity complexes $\textrm{Act}_<(M)$ under the different linear orders $<$ of the ground set. This led them to consider and describe the complexes $\textrm{Act}_<(M)$. 

The ideals $\textrm{in}_<\widetilde{L}$ are shown to be Cohen-Macaulay in \cite{ArdilaBoocher}, and the authors asked the stronger question: Are the external activity complexes $\textrm{Act}_<(M)$  shellable? %(This stronger combinatorial-topological statement about $\textrm{Act}_<(M)$ would imply that algebraic result about its Stanley-Reisner ideal $\textrm{in}_<\widetilde{L}$.) 
The purpose of this note is to answer this question affirmatively.

\bigskip

\noindent \textbf{\textsf{Our results.}}
The facets of $\textrm{Act}_<(M)$ are indexed by the bases $\B$ of $M$, and \cite{ArdilaBoocher} suggested a possible connection between $\textrm{Act}_<(M)$ and LasVergnas's \emph{internal order} $<_{int}$ on $\B$. \cite{LasVergnas} Suprisingly, we find that it is the \emph{external/internal order} $<_{ext/int}$ on $\B$, also defined in \cite{LasVergnas}, which plays a key role. Our main result is the following:

\begin{theorem}\label{main}
Let $M = (E, \mathcal{B})$ be a matroid, and let $<$ be a linear order on the ground set $E$. Any linear extension of LasVergnas's external/internal order $<_{ext/int}$ of $\B$ induces a shelling of the external activity complex $\textrm{Act}_<(M)$.
\end{theorem}

As a corollary we obtain that these orders also shell the independence complex $IN(M)$, and in fact we show a stronger statement.

\begin{theorem}\label{intshell} Any linear extension of the internal order $<_{int}$ gives a shelling order of the independence complex $IN(M)$.
\end{theorem}

These theorems are as strong as possible in the context of LasVergnas's active orders. 
We also obtain the following enumerative corollary. 

\begin{theorem}\label{hvector}
The $h$-vector of $\textrm{Act}_<(M)$ equals the $h$-vector of $M$.
\end{theorem}

It is easy to see that $\textrm{Act}_<(M)$ is a cone, and hence trivially contractible. It is more interesting to study the \emph{reduced external activity complex} ${\textrm{Act}}^\bullet_<(M)$, obtained by removing all the cone points of $\textrm{Act}_<(M)$. Our main topological result is the following.
%Also identify $M$ with its simplicial complex of independent sets, which is known to be a wedge of spheres. \cite{bjorner}

\begin{theorem} \label{topo} %Assume that $M$ is coloop free. 
Let $M$ be a matroid and $<$ be a linear order on its ground set. The reduced external activity complex ${\textrm{Act}}^\bullet_<(M)$ is contractible if $M$ contains $U_{3,1}$ as a minor, and a sphere otherwise.
\end{theorem}

%
%\begin{theorem}
%The reduced external activity complex $\widetilde{\textrm{Act}$ is contractible unless the circuits of $M$ are pairwise disjoint, in which case it is homeomorphic to $M$.%a wedge of \comment{cuantas?} spheres?
%\end{theorem}

% $\Omega(M)$ However, we show that even after removing all cone points, the resulting complex $\Omega(M)$ is still contractible. To show this the first step is to prove that $\textrm{Act}_<(M)$ is shellable. Then, using the shelling order, we compute the $h-$vector, which turns out to be the same as the $h-$vector of $M$. This is enough to deduce that $\Omega(M)$ is contractible for all matroids except when all its circuits are disjoint. In that case $\Omega(M)$ is homemorphic to $M$.\\
%For all other matroids the situation is the following: t
In Proposition \ref{subcomplex} we will see there is an embedding of the independence complex $IN(M)$ in ${\textrm{Act}}^\bullet_<(M)$, and both complexes have the same $h$-vector. 
If $M$ is coloopless its independence complex is homotopy equivalent to a wedge of $|\mu(M^*)|$ spheres, while the external activity complex is contractible or a sphere. 
%A matroid $M$ is contractible if and only if it is a cone (Theorem 3.4 \cite{stanley}). By contrast, $\widetilde{\textrm{Act}$ is always contractible and never a cone. 
Thus ${\textrm{Act}}^\bullet_<(M)$ can be seen as a topologically simpler model than $IN(M)$ for the matroid  $M$.

%Thus $\underline{\textrm{Act}}_<(M)$ may be seen as an alternative simplicial complex that models the matroid $M$, and is topologically simpler than the matroid complex $IN(M)$. 

The paper is organized as follows. In Section \ref{sec:background} we introduce the necessary definitions and preliminaries. In Section \ref{sec:example} we carry out an example in detail, and show that the hypotheses of Theorems \ref{main} and \ref{intshell} are best possible. 
In Section \ref{sec:shell} we prove our main Theorem \ref{main} on the shellability of the external activity complex $\textrm{Act}_<(M)$, and Theorem \ref{intshell}, which gives many new shellings of the independence complex $IN(M)$. In Section \ref{sec:h} we show that $\textrm{Act}_<(M)$ and $IN(M)$ have the same $h$-vector. Finally, in Section \ref{sec:topology}, we describe the topology of the reduced external activity complex in Theorem \ref{topo}. %In Section \ref{sec:example} we carry out an example carefully. It may be useful to refer to it 

%---------------------------------------------------------------------------------Background------------------------------------------------------------------------
\section{\textsf{Preliminaries}}\label{sec:background}

In this section we collect the background information on matroids and shellability that we will need to prove our results.

\subsection{\textsf{Matroids}}

\textbf{\textsf{Basic definitions.}}
A {\it simplicial complex} $\Delta = (E, \I)$ is a pair where $E$ is a finite set and $\I$ is a non empty family of subsets of $E$, such that if $A\in \I$ and $B\subset A$, then $B\in A$. Elements of $\I$ are called {\it faces} of the complex. The maximal elements of $\I$ are called {\it facets}. A complex is said to be {\it pure} if all facets have the same number of elements.

The following is one of many equivalent ways of defining a matroid:

\begin{defi}
A \emph{matroid} $M = (E, \I)$ is a simplicial complex such that the restriction of $M$ to any subset of $E$ is pure. 
\end{defi}

Since there are several simplicial complexes associated to $M$, we will denote this one $IN(M)=(E,\I)$. It is often called the \emph{independence complex} of $M$.

The two most important motivating examples of matroids are the following.
\begin{itemize}
\item (Linear Algebra) Let $E$ be a set of vectors in a vector space, and let $\I$ consist of the subsets of $E$ which are linearly independent. Then $(E,\I)$ is a \emph{linear} matroid.

\item (Graph Theory) Let $E$ be the set of edges of an undirected graph $G$, and let $\I$ consist of the sets of edges which contain no cycle. Then  $(E, \I)$ is a \emph{graphical} matroid. 
\end{itemize}

For any matroid $M=(E,\I)$, it is customary to call the sets in $\I$ \emph{independent}.
The facets of a matroid are called {\it bases}. The set of all bases is denoted $\cal B$. 

\begin{ex}The simplest example of a matroid is the uniform matroid $U_{n,k}$, whose ground set is $[n]$ and whose independent sets are all the subsets of $[n]$ of cardinality at most $k$. The uniform matroid $U_{3,1}$ is going to play an important role later.
\end{ex}
%The bases of a matroid satisfy the following property.
% It satisfies the following exchange property (Lemma 1.2.2 \cite{oxley}):
%\begin{lem}[Basis exchange]\label{exchange} For any two bases $B, C$ and any $i\in B-C$, there exist $j\in C-B$ such that $B-i\cup j$ is a basis. 
%\end{lem}
%We will also use an apparently stronger (but in fact equivalent) property:

%\begin{lem}[Symmetric Exchange Property]\label{symmexch} For any two bases $B, C$ and any $b\in B-C$, there exist $c\in C-B$ such that $B-b\cup c$ and $C-c\cup b$ are bases.\end{lem}

%\begin{proof} Let $c_1, \dots, c_m$ be the elements of $C-(\Circ(i,C)\cup B)$. Let $C_0 = C$ and let $C_k = C_{k-1}-c_k\cup b_k$ be a basis that results from applying the exchange axiom to $C_{k-1}-c_k$ with $B$. We claim that $i\notin C_k$ and $\Circ(i, C_k) = \Circ(i, C)$. The proof is by induction on $k$. The base case is trivial. If the statement holds for $C_{k-1}$ then $b_k \not= i$ because $c_k\notin \Circ(i,C_{k-1}) = \Circ(i,C)$, thus $i\not\in C_k$. Also $C(i,C_{k-1}) \subseteq C_{k-1}\cup i - c_k \subseteq C_k \cup i$, thus $\Circ(i, C_{k-1}) = \Circ(i, C_k)$ as desired. Let $C':=C_m$. Then $C_m$ is a basis such that $C'-B = \Circ(i, C)-B$, $i\not\in C'$. By the exchange axiom, there is $j\in C'-B$ such that $B-i\cup j$ is a basis. Since $j\in \Circ(i,C)$ we have that $C-j\cup i$ is a basis, so $j$ is the desired element. \end{proof}  

The minimal non-faces of $M$, that is, the minimal dependent sets, are called \emph{circuits}. The circuits of a matroid have a special structure \cite[Lemma 1.1.3]{oxley}: 

\begin{lem}[Circuit Elimination Property]\label{circuitax} If $\gamma_1$ and $\gamma_2$ are circuits of a matroid and $c\in \gamma_1\cap \gamma_2$, then there is a circuit $\gamma_3$ that is contained in $\gamma_1\cup \gamma_2 - c$. 
\end{lem}

\bigskip
\noindent \textsf{\textbf{Duality.}}
Matroids have a notion of duality which generalizes orthogonal complements in linear algebra and dual graphs in graph theory. 

Let $M$ be a matroid with bases $\B$. Then the set
\[
\B^* = \{ E - B \, : \, B \textrm{ is a basis of } M\}
\]
is the collection of bases of a matroid $M^*=(E,\B^*)$, called the \emph{dual matroid} $M^*$. 
The circuits of the dual matroid $M^*$ are called the \emph{cocircuits} of $M$.

\bigskip
\noindent \textsf{\textbf{Deletion, contraction, and minors.}} 
%Let $M$ be a matroid whose ground set is $E$. 
We say that an element $e\in E$ is a \textit{loop} of a matroid $M$ if it is contained in no basis; that is, if $\{e\}$ is a dependent set. Dually, $e$ is a \textit{coloop} if it is contained in every basis of $M$.

The \emph{deletion} $M\backslash e$ of a non-coloop $e \in E$ is the matroid on $E-e$ whose bases are the bases of $M$ that do not contain $e$. We also call this the \emph{restriction} of $M$ to $E-e$. Dually, the \emph{contraction} $M\slash e$ of a non-loop $e \in E$ is the matroid on $E-e$ whose bases are the subsets $B$ of $E-e$ such that $B\cup e$ is a basis of $M$.

It is easy to see that any sequence of deletions and contractions of different elements commutes.
We say that a matroid $M'$ is a \emph{minor} of a matroid $M$ if $M'$ is isomorphic to a matroid obtained from $M$ by performing a sequence of deletions and contractions. 
%In other words $M'$ is a minor of $M$ if and only if one of the following happens: 
%\begin{enumerate}
%\item $M' \cong M$. 
%\item There exists $e\in E(M)$ that is not a coloop and $M'$ is a minor of $M\backslash e$. 
%\item There exists $e\in E(M)$ that is not a loop and $M'$ is a minor of $M'\slash e$. 
%\end{enumerate}

\bigskip

\noindent \textsf{\textbf{Fundamental circuits and cocircuits.}}
Given a basis $B$ and an element $e\in E-B$ there is a unique circuit contained in $B\cup e$, 
called the \emph{fundamental circuit} of $e$ with respect to $B$. It is given by
\[
\Circ(B,e) = \left\{ x\in E\, : \,  B\cup e - x \in {\cal B} \right\}.
\]

 Given a basis $B$ and an element $i \in B$ there is a unique cocircuit disjoint with $B-i$, called the \emph{fundamental cocircuit} of $i$ with respect to $B$. It is given by
\[
\Cocirc(B,i) = \left\{ x\in E\, : \,  B\cup x - i \in {\cal B} \right\}.
\]
Note that the cocircuit $\Cocirc(B,i)$ in $M$ equals the circuit $\Circ(E-B,i)$ in the dual $M^*$.

%\subsection{\textsf{Matroids: Basis activities and the external activity complex}}

\bigskip

\noindent \textsf{\textbf{Basis activities.}}
Let $<$ be a linear order on the ground set $E$. For a basis $B$, define the sets: 
\begin{eqnarray*} 
EA(B) &=& \left\{ e\in E-B\, : \,  \textrm{min} \left(\Circ(B,e)\right)=e\right\}
\\
EP(B) &=& \left\{ e\in E-B\, : \,  \textrm{min} \left(\Circ(B,e)\right) \neq e\right\}
%EP(B) &=& (E-B) - EA(B)
\end{eqnarray*}
The elements of $EA(B)$ and $EP(B)$ are called \textit{externally active} and \textit{externally passive} with respect to $B$, respectively. Note that $EA(B) \uplus EP(B) = E-B$, where $\uplus$ denotes a disjoint union.

Dually, let
\begin{eqnarray*} 
IA(B) &=& \left\{ i\in B\, : \,  \textrm{min} \left(\Cocirc(B,i)\right)=i\right\}
\\
IP(B) &=& \left\{ i \in B\, : \,  \textrm{min} \left(\Cocirc(B,i)\right) \neq i\right\}
%EP(B) &=& (E-B) - EA(B)
\end{eqnarray*}
The elements of $IA(B)$ and $IP(B)$ are called \textit{internally active} and \textit{internally passive} with respect to $B$, respectively. Note that $IA(B) \uplus IP(B) = B$. Also note that the internally active/passive elements with respect to basis $B$ in $M$ are the externally active/passive elements with respect to basis $E-B$ in $M^*$.

The following elegant result of Tutte \cite{Tutte} (for graphs) and Crapo \cite{Crapo} (for arbitrary matroids) underlies many of the results of \cite{ArdilaBoocher} and this paper.
\begin{theorem}\cite[Proposition 5.12]{Crapo}\label{crapo}
Let $M$ be a matroid on the ground set $E$ and let $<$ be a linear order on $E$. 
\begin{enumerate}
\item
Every subset $A$ of $E$ can be uniquely written in the form $A=B \cup X - Y$ for some basis $B$, some subset $X \subseteq EA(B)$, and some subset $Y \subseteq IA(B)$. 
Equivalently, the intervals $[B - IA(B), B \cup EA(B)]$ form a partition of the poset $2^E$ of subsets of $E$ ordered by inclusion.
\item
Every independent set $I$ of $E$ can be uniquely written in the form $I=B - Y$ for some basis $B$ and some subset $Y \subseteq IA(B)$. 
Equivalently, the intervals $[B - IA(B), B]$ form a partition of the independence complex $IN(M)$. %\comment{Esto es una consecuencia directa de que lex es un shelling order con resteriction set $IP(B) = B- IA(B)$. Lo mencionamos? }
\end{enumerate}

\end{theorem}

The \emph{Tutte polynomial} of $M$ is
\[
T_M(x,y) = \sum_{B \textrm{ basis}} x^{|IA(B)|} y^{|EA(B)|}.
\]
It follows from the work of Crapo and Tutte \cite{Crapo, Tutte} that this polynomial does not depend on the chosen order $<$. The Tutte polynomial is the most important matroid invariant, because it answers an innumerable amount of questions about the combinatorics, algebra, geometry, and topology of matroids and related objects. For more information, see \cite{BrylawskiOxley}.

\bigskip

\noindent \textsf{\textbf{The external activity complex.}} Let $M$ be a matroid on $E$. 
Let $\overline{E} = \{\overline{e} \, : \, e \in E\}$ be a second copy of $E$, and let $[[E]] = E \uplus \overline{E}$. This set of size $2|E|$ will be the ground set of the external activity complex of $M$. 
For each subset $S \subseteq E$ we write $\overline{S} := \{ \overline{s} \, | \, s \in S\} \subset \overline{E}$. Therefore, each subset of $[[E]]$ can be written uniquely in the form $S_1 \cup \overline{S_2}$ for $S_1, S_2 \subseteq E$.

Our main object of study is the following.

\begin{teo} \cite{ArdilaBoocher} 
Let $M = (E, \B)$ be a matroid and let $<$ be a linear order on $E$. $M$. There is a simplicial complex called the {\it external activity complex} $\textrm{Act}_<(M)$ on ground set $[[E]]$ such that
\begin{enumerate}
\item
The facets are
$
F(B):=B\cup EP(B) \cup \overline{B\cup EA(B)}
$
for every basis $B\in {\cal B}$.
\item
The minimal non-faces are $S(\gamma) = c \cup \overline{\gamma-c}$
for every circuit $\gamma$, where $c$ is the $<$-smallest element of $\gamma$.
\end{enumerate}
%\begin{enumerate}
%\item
%The facets are
%\[
%F(B):=B\cup EP(B) \cup \overline{B\cup EA(B)}
%\]
%for every basis $B\in {\cal B}$.
%\item
%The minimal non-faces are
%\[
%S(\gamma) = c \cup \overline{\gamma-c}
%\]
%for every circuit $\gamma$, where $c$ is the $<$-smallest element of $\gamma$.
%\end{enumerate}
\end{teo}

The complement of the facet $F(B)$ in $[[E]]$ is %denoted by $G(B)$ and is given by 
$G(B) = EA(B) \cup \overline{EP(B)}$.

\bigskip

\noindent \textsf{\textbf{Las Vergnas's three active orders.}}
Given a matroid $M=(E,\B)$ and a total order $<$ on the ground set of $M$, LasVergnas introduced the following three \emph{active orders}. In each case, he proved that there are several equivalent definitions.

\begin{defi} \label{ext}
The \emph{external order} $<_{ext}$ on $\B$ is characterized by the following equivalent properties for two bases $A$ and $B$:

1. $A \leq_{ext} B$,

2. $A \subseteq B \cup EA(B)$,

3. $A \cup EA(A) \subseteq B \cup EA(B)$,

4. $B$ is the lexicographically largest basis contained in $A \cup B$.

\noindent 
This poset is graded with $r(B) = |EA(B)|$. Adding a minimum element turns it into a lattice.
\end{defi}

%
%
%\begin{defi} \label{ext}
%The \emph{external order} $<_{ext}$ on $\B$ is characterized by the following equivalent properties for two bases $A$ and $B$:
%\begin{enumerate}
%\item
%$A \leq_{ext} B$,
%\item
%$A \subseteq B \cup EA(B)$,
%\item
%$A \cup EA(A) \subseteq B \cup EA(B)$,
%\item
%$B$ is the lexicographically largest basis contained in $A \cup B$.
%\end{enumerate}
%This poset is graded with $r(B) = |EA(B)|$. Adding a minimum element turns it into a lattice.
%\end{defi}

\begin{defi} \label{int}
The \emph{internal order} $<_{int}$ on $\B$ is characterized by the following equivalent properties for two bases $A$ and $B$:

1. $A \leq_{int} B,$

2. $A  - IA(A) \subseteq B,$

3. $A -IA(A) \subseteq B - IA(B),$

4. $A$ is the lexicographically smallest basis containing $A \cap B$.

\noindent This poset is graded with $r(B) = r-|IA(B)|$. Adding a maximum element turns it into a lattice.
\end{defi}

%\begin{defi} \label{int}
%The \emph{internal order} $<_{int}$ on $\B$ is characterized by the following equivalent properties for two bases $A$ and $B$:
%\begin{enumerate}
%\item
%$A \leq_{int} B,$
%\item
%$A  - IA(A) \subseteq B,$
%\item
%$A -IA(A) \subseteq B - IA(B),$
%\item
%$A$ is the lexicographically smallest basis containing $A \cap B$.
%\end{enumerate}
%This poset is graded with $r(B) = r-|IA(B)|$. Adding a maximum element turns it into a lattice.
%\end{defi}

The internal and external orders are consistent in the sense that $A \leq_{int} B$ and $B \leq_{ext} A$ imply $A=B$. Therefore the following definition makes sense.

\begin{defi} \label{extint}
The \emph{external/internal order} $<_{ext/int}$ is the weakest order which simultaneously extends the external and the internal order. It is characterized by the following equivalent properties for two bases $A$ and $B$:

1. $A \leq_{ext/int} B,$

2. $IP(A) \cap EP(B) = \emptyset,$

\noindent This poset is a lattice. It is not necessarily graded.
\end{defi}

Note that Theorem \ref{ext}.4 and \ref{int}.4 imply the following.

\begin{prop} \label{lexintext}
The lexicographic order $<_{lex}$ on $\B$ is a linear extension of the three posets $<_{int}, <_{ext},$ and $<_{ext/int}$. In symbols, any of $A<_{int} B$, $A<_{ext} B$ or $A <_{ext/int} B$ implies $A <_{lex} B$.
\end{prop}

\subsection{\textsf{Shellability and the $h$-vector}}

\noindent \textsf{\textbf{Shellability.}} Shellability is a combinatorial condition on a simplicial complex that allows us to describe its topology easily. A simplicial complex is shellable if it can be built up by introducing one facet at a time, so that whenever we introduce a new facet, its intersection with the previous ones is pure of codimension 1. More precisely:

\begin{defi}
Let $\Delta$ be a pure %$(d-1)$-dimensional 
simplicial complex. 
A \emph{shelling order} is an order of the facets $F_1, \dots F_k$ such for every $i<j$ there exist $k<j$ and $f\in F_j$ such that $F_i \cap F_j \subseteq F_k\cap F_j = F_j-f$. 
If a shelling order exists, then we call $\Delta$ \emph{shellable}.
\end{defi}
%
%We say that a complex $\Delta$ is {\it shellable} if it has a shelling order. Informally, this means that we can build up the complex by introducing one facet at a time; and whenever we introduce a new facet $F_j$, its intersection with the previous ones is pure of codimension 1.

Given a shelling order and a facet $F_j$, there is a subset $\RR(F_j)$ such that for every $A \subseteq F_j$, we have 
$A \nsubseteq F_i$ for all $i<j$ if and only if $\RR(F_j) \subseteq A$. 
%
%the following are equivalent: \begin{itemize} 
%\item $A$ is not a subset of $F_i$ for every $i<j$. 
%\item $\RR(F_j) \subseteq A$. 
%\end{itemize}
Equivalently, when we add facet $F_j$ to the complex, the new faces that we introduce are precisely those in the interval  $[\RR(F_j), F_j]$. The set $\RR(F_j)$ is called the {\it restriction set} of $F_j$ in the shelling.

\medskip

\noindent \textsf{\textbf{The $f$-vector and $h$-vector.}} The \emph{$f$-vector} of a $(d-1)$-dimensional simplicial complex $\Delta$ is $(f_0, \ldots, f_d)$ where $f_i$ is the number of faces of $\Delta$ of size $i$. The \emph{$h$-vector} $(h_0, \ldots, h_d)$ is an equivalent way of storing this information; it is defined by the relation
\[
f_0(x-1)^d + f_1(x-1)^{d-1} + \cdots + f_d(x-1)^0 = h_0x^d + h_1 x^{d-1} + \cdots + h_dx^0.
\]
This polynomial is also known as the \emph{shelling polynomial} $h_\Delta(x)$, due to the following description of the $h$-vector for shellable complexes.

\begin{prop}\label{shellingh} \cite[Proposition 7.2.3]{bjorner}
If $F_1, \ldots, F_k$ is a shelling order for a $(d-1)$-dimensional simplicial complex $\Delta$, then 
\[
h_i:= \left| \, \{j \, : \,   |\RR(F_j)| = i\}\right|.
\]
\end{prop}

% If $\dim \Delta = d-1$, then the \emph{$h$-vector} $(h_0, h_1, \dots h_d)$ is defined by: \[h_i:= |\{j\, : \,   |\RR(F_j)| = i\}|\]

Note that it is not clear a priori that these numbers should be the same for any shelling order.

Understanding the topology of a shellable simplicial complex is easy once we know the last entry of the $h$-vector, thanks to the following result.

%\comment{No me gusta la referencia. DespuŽs de discutir con Isabella, llegamos a la conclusion de que por ahora lo mejor es poner el paper de non-pure shellability de bjorner y wachs. Ahi ellos mencionan este teorema sin referencia, pero dan una prueba para complejos no necesariamente puros.}

\begin{theorem}\cite[Theorem 12.2(2)]{Kozlov} \label{contraible}Any geometric realization of a $(d-1)$-dimensional shellable simplicial complex $\Delta$ is homotopy equivalent to a wedge  of $h_d$ spheres of dimension $d-1$. In particular, if $h_d = 0$, then every geometric realization of $\Delta$ is contractible. 
\end{theorem}
An important property for matroids is their shellability: 

\begin{theorem}\cite[Theorem 7.3.3]{bjorner} \label{matroidshell}
The lexicographic order $<_{lex}$ on the bases of a matroid $M$ gives a shelling order of the independence complex $IN(M)$. Furthermore, the restriction set of a basis $B$ in this shelling order is given by $IP(B)$. 
\end{theorem}

A straightforward consequence of the previous theorem is that the internal order poset is equal to the poset of bases of $M$ where the order is given by inclusion of restriction sets of the lexicographic shelling order.

\section{\textsf{Example}}\label{sec:example}

Before proving our theorems, we illustrate them in an example. 
Consider the graphical matroid given by the graph of Figure \ref{matroide}. Its bases are all the $3$-subsets of $[5]$ except $\{1,2,3\}$ and $\{1,4,5\}$.
Under the standard order $1<2<3<4<5$ on the ground set, Table \ref{table1} records the basis activity of the various bases.

\begin{figure}[h!]
\centering
\includegraphics[scale=0.45]{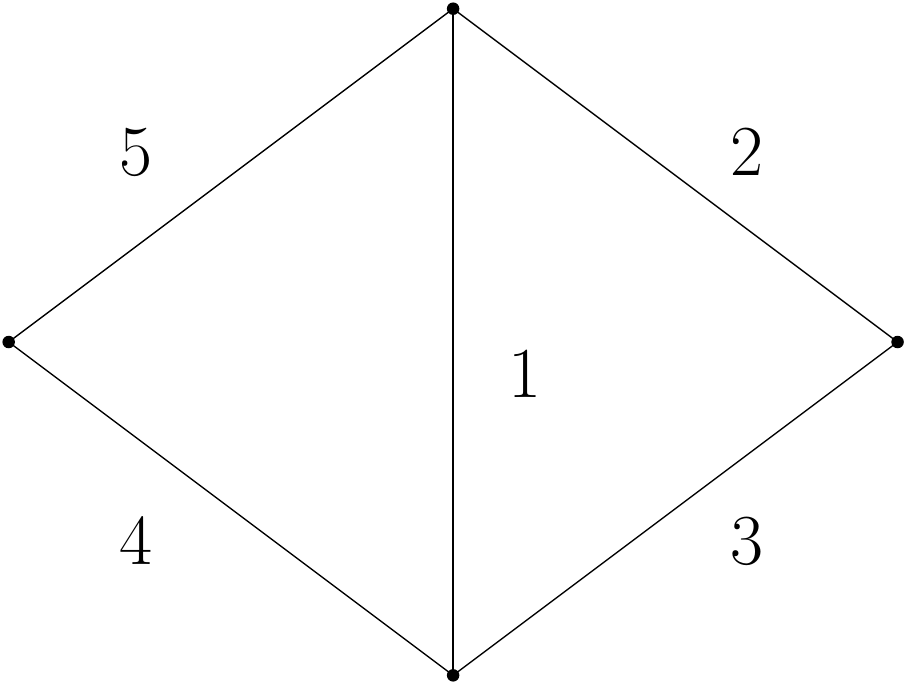}
\caption{A graphical matroid.}
\label{matroide}
\end{figure}

%; it lists each basis $B$, 
%the externally passive elements $EP(B)$,
%the externally active elements $EA(B)$,
%the internally passive elements $IP(B)$, and
%the internally active elements $IA(B)$. 

%We have the following data for the independent complex $IN(M)$:
%\begin{center}
%\begin{tabular}{| l | c | c | c |}
%\hline
%Bases & Ext. Passive & Ext. Active & Restriction Set \\ \hline
%$\left\{1,2,4\right\}$ & $\left\{3,5\right\}$ & $\emptyset$ & $\emptyset$ \\ \hline
%$\left\{1,2,5\right\}$ & $\left\{4,5\right\}$ & $\emptyset$ &  $\left\{5\right\}$ \\ \hline
%$\left\{1,3,4\right\}$ & $\left\{2,5\right\}$ & $\emptyset$ &  $\left\{3\right\}$ \\ \hline
%$\left\{1,3,5\right\}$ & $\left\{2,4\right\}$ & $\emptyset$ &  $\left\{3,5\right\}$ \\ \hline
%$\left\{2,3,4\right\}$ & $\left\{5\right\}$ & $\left\{1\right\}$  &  $\left\{2,3\right\}$\\ \hline
%$\left\{2,3,5\right\}$ & $\left\{4\right\}$ & $\left\{1\right\}$ &  $\left\{2,3,5\right\}$ \\ \hline
%$\left\{2,4,5\right\}$ & $\left\{3\right\}$ & $\left\{1\right\}$ &  $\left\{4,5\right\}$ \\ \hline
%$\left\{3,4,5\right\}$ & $\emptyset$ & $\left\{1,2\right\}$ &  $\left\{3,4,5\right\}$ \\ \hline
%\end{tabular}
%\end{center}
\begin{table}[h]
\begin{center}
\begin{tabular}{| c || c | c | c | c |}
\hline
$B$ & $EP(B)$ & $EA(B)$ & $IP(B)$ & $IA(B)$\\
\hline
%Bases & Ext. Passive & Ext. Active & Restriction Set \\ \hline
$124$ & $35$ & $\emptyset$ & $\emptyset$ &124 \\ \hline
$125$ & $45$ & $\emptyset$ &  $5$ & 12\\ \hline
$134$ & $25$ & $\emptyset$ &  $3$ & 14 \\ \hline
$135$ & $24$ & $\emptyset$ &  $35$& 1 \\ \hline
$234$ & $5$ & $1$  &  $23$& 4 \\ \hline
$235$ & $4$ & $1$ &  $235$ & $\emptyset$ \\ \hline
$245$ & $3$ & $1$ &  $45$ & 2 \\ \hline
$345$ & $\emptyset$ & $12$ &  $345$ & $\emptyset$\\ \hline
\end{tabular}
 \end{center}
\caption{The bases $B$ together with their sets of externally passive, externally active, internally passive, and internally active elements.
 \label{table1}}
 \end{table}

The resulting internal, external, and external/internal orders $<_{int}, <_{ext}, <_{ext/int}$ are shown in Figure \ref{orders}. By Theorems \ref{ext}, \ref{int}, and \ref{extint}, these three orders are isomorphic to the three families of sets 
$\{B \cup EA(B) \, : \, B \textrm{ basis}\}$,
$\{B-IA(B) \, : \, B \textrm{ basis}\}$, and
$\{B \cup EA(B) -IA(B) \, : \, B \textrm{ basis}\}$,
partially ordered by containment.

\begin{figure}[h!]
\centering
\includegraphics[scale=0.4]{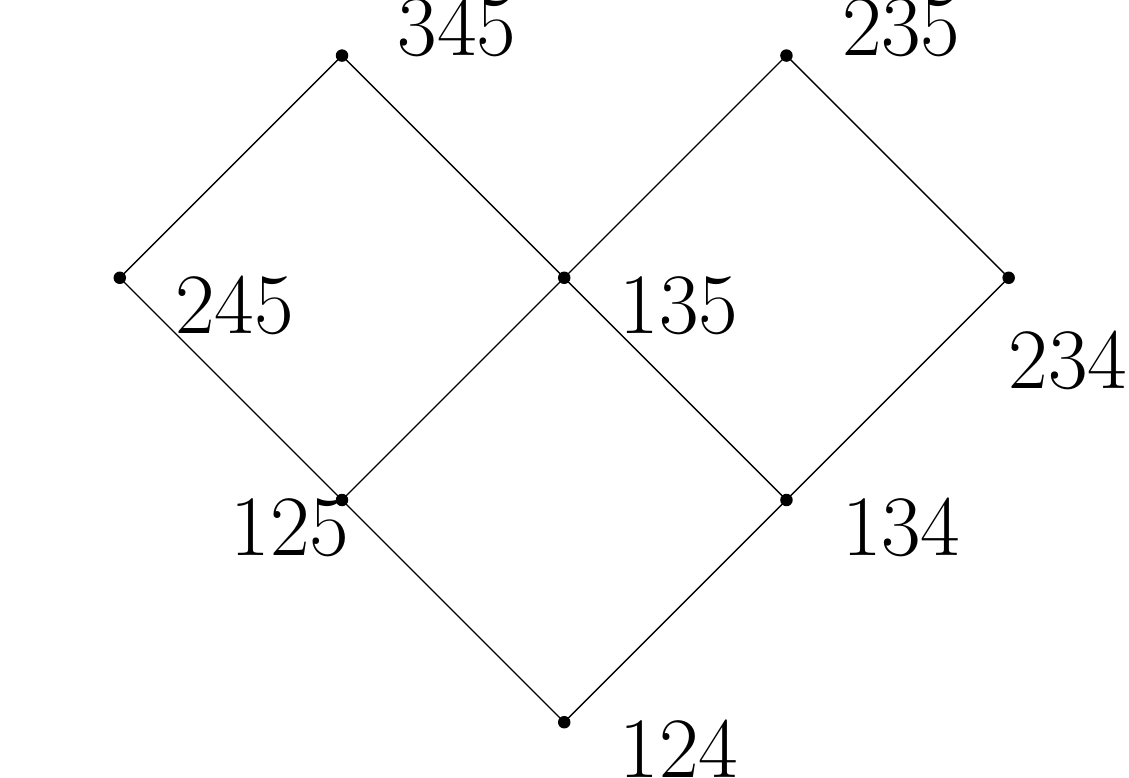} \quad
\includegraphics[scale=0.4]{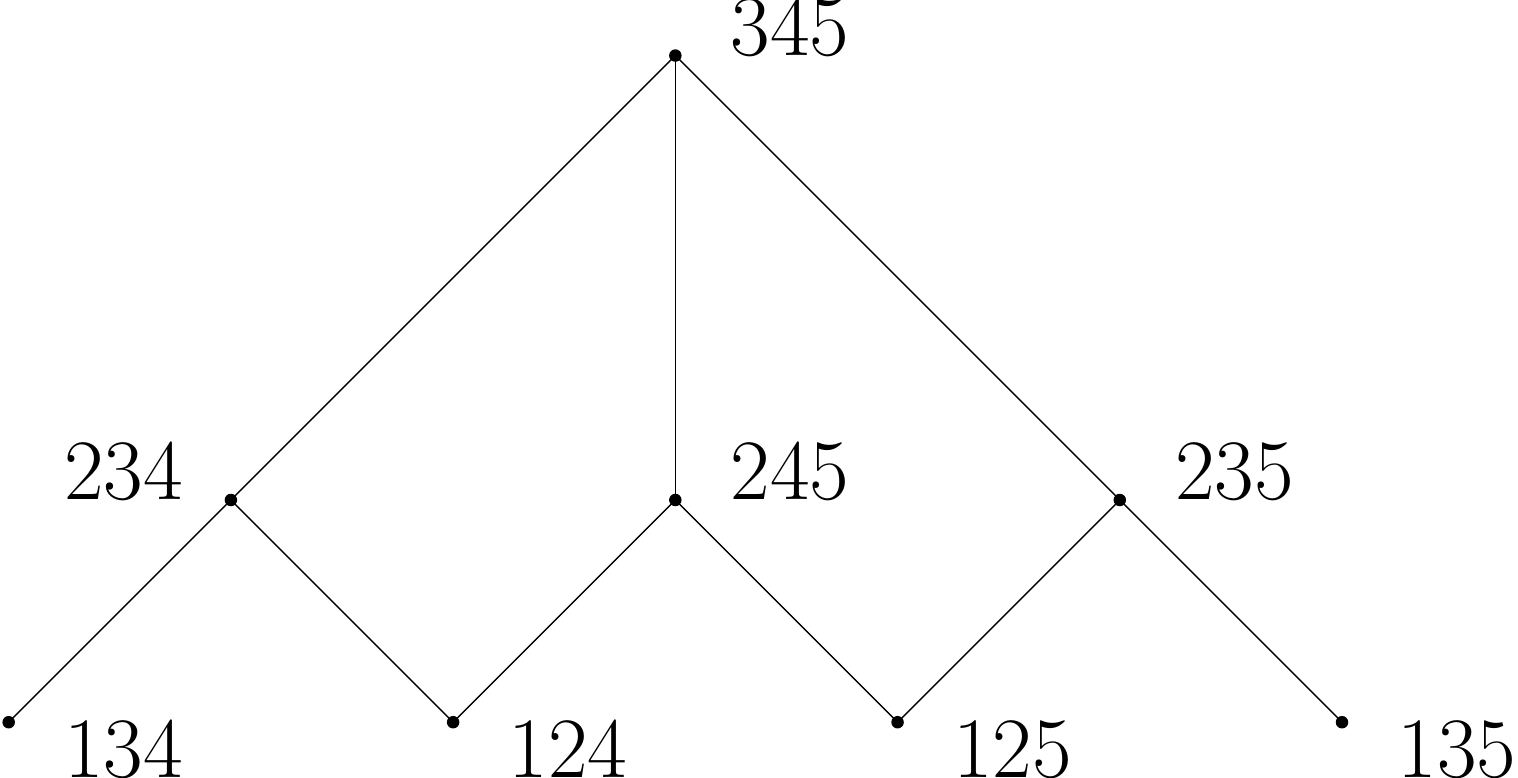} \quad
\includegraphics[scale=0.4]{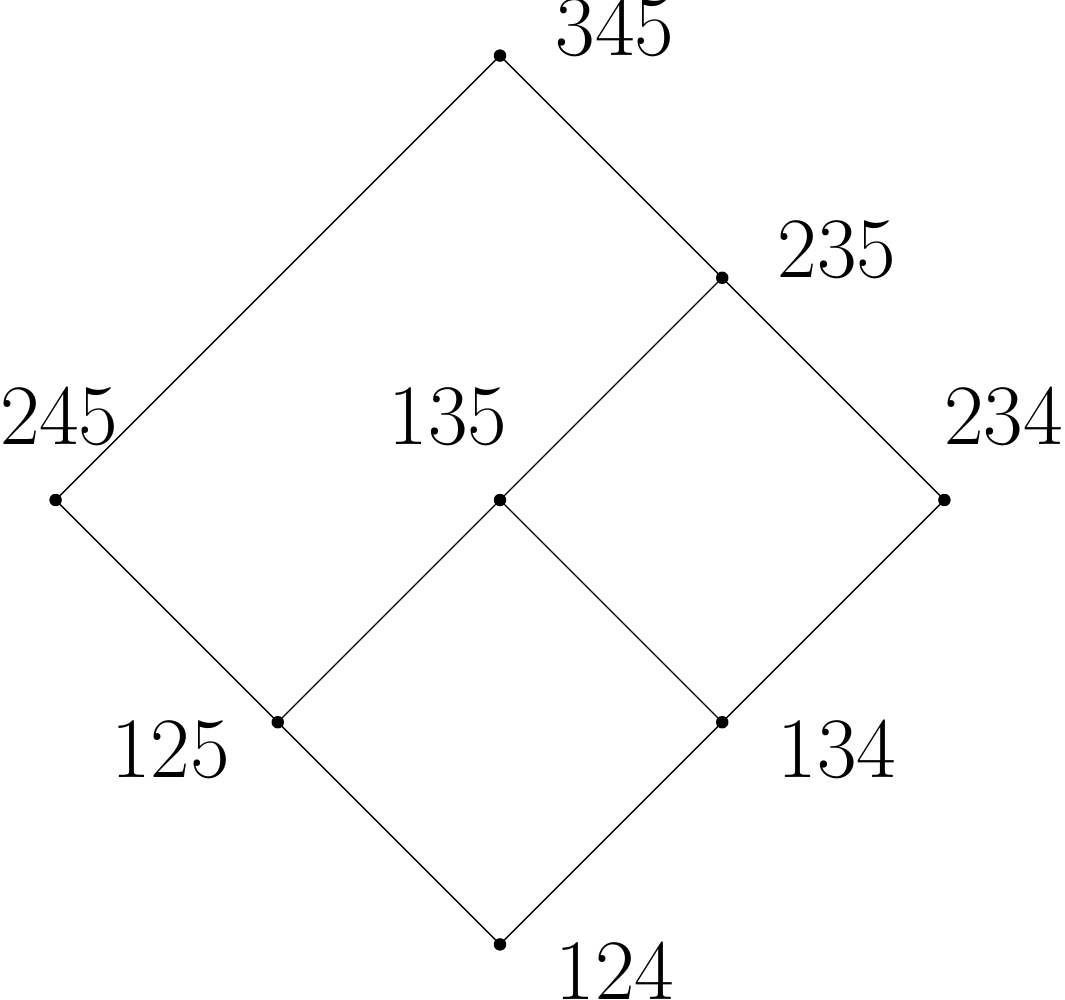}
\caption{The active orders $<_{int}, <_{ext}$, and $<_{ext/int}$, respectively.}
\label{orders}
\end{figure}
%
%\begin{figure}[h!]
%\centering
%\includegraphics[scale=0.6]{ExtInt.pdf}
%\caption{The external/internal order $<_{ext/int}$.}
%\label{intextorder}
%\end{figure}

Table \ref{table1} lists the bases in lexicographic order $<_{lex}$, and this is a shelling order for the independence complex $IN(M)$ by Theorem \ref{matroidshell}. The restriction set for each basis $B$ is $\RR(B) = IP(B)$. For example, when we add facet $134$ in the third step of the shelling, this means that the new faces that appear are the four sets in the interval $[\RR(134),134] = [3,134]$; that is, faces $3, 13, 34,$ and $134$.

\bigskip

Our goal is to shell the external activity complex $\textrm{Act}_<(M)$ whose facets, listed in Table \ref{table2}, are the sets $F(B) = B \cup EP(B) \cup \overline{B \cup EA(B)}$. Since $\overline{1},3,4,$ and $5$ are in all facets of $\textrm{Act}_<(M)$, we remove them, and shell the resulting \emph{reduced external activity complex} ${Act}^\bullet_<(M)$. 
Our main result, Theorem \ref{main}, states that any linear extension of the external/internal order $<_{ext/int}$ gives a shelling order for this complex. For example, we may again consider the lexicographic order,
%\[
%124, 125, 134, 135, 234, 235, 245, 345
%\]
which is indeed a linear extension of $<_{ext/int}$.

 \begin{table}[h]
\begin{center}
\begin{tabular}{|c|c|c|c|}
\hline
$B$ & $F(B)$ & ${F(B)}^\bullet$ & $\RR(F(B))$  \\
%& & & 
%\vspace{-.1cm}\\
\hline
%&&&\\
124 & $12345\overline{1}\overline{2}\overline{4}$ & $12\overline{2}\overline{4}$ & $\emptyset$ \\ 
125 & $12345\overline{1}\overline{2}\overline{5}$ & $12\overline{2}\overline{5}$ & $\overline{5}$ \\ 
134 & $12345\overline{1}\overline{3}\overline{4}$ & $12\overline{3}\overline{4}$ & $\overline{3}$ \\ 
135 & $12345\overline{1}\overline{3}\overline{5}$ & $12\overline{3}\overline{5}$ & $\overline{3}\overline{5}$ \\ 
234 & $2345\overline{1}\overline{2}\overline{3}\overline{4}$ & $2\overline{2}\overline{3}\overline{4}$ & $\overline{2}\overline{3}$ \\ 
235 & $2345\overline{1}\overline{2}\overline{3}\overline{5}$ & $2\overline{2}\overline{3}\overline{5}$ & $\overline{2}\overline{3}\overline{5}$ \\ 
245 & $2345\overline{1}\overline{2}\overline{4}\overline{5}$ & $2\overline{2}\overline{4}\overline{5}$ & $\overline{4}\overline{5}$ \\ 
345 & $345\overline{1}\overline{2}\overline{3}\overline{4}\overline{5}$ & $\overline{2}\overline{3}\overline{4}\overline{5}$ & $\overline{3}\overline{4}\overline{5}$ \\ 
\hline
\end{tabular}
\end{center}
\caption{The bases $B$ of $M$, the corresponding facets $F(B)$ and ${F(B)}^\bullet$ of 
${\textrm{Act}}_<(M)$ and ${\textrm{Act}}^\bullet_<(M)$, and their (shared) restriction set $\RR(F(B))$ in the shelling.\label{table2}}
\end{table}

For each basis $B$, 
Table \ref{table2} lists the corresponding facet $F(B)$ of $\textrm{Act}_<(M)$, 
 the corresponding facet ${F(B)}^\bullet$ of ${\textrm{Act}}^\bullet_<(M)$, and the restriction set of the facet $F(B)$ in the shelling. 
This restriction set is $\RR(F(B)) = \overline{IP(B)}$. For example, when we add facet $12\overline{3}\overline{4}$ to the complex ${\textrm{Act}}^\bullet_<(M)$ in the third step of the shelling, the new faces that appear are the eight sets in the interval $[\RR(12\overline{3}\overline{4}),12\overline{3}\overline{4}] = [\overline{3}, 12\overline{3}\overline{4}]$.

\medskip

Notice that we can embed $IN(M)\longrightarrow {\textrm{Act}}^\bullet_<(M)$ by sending $1\to 1, 2\to \overline{2}, 3\to \overline{3}, 4\to \overline{4}, 5\to \overline{5}$. The latter complex has the same $h$-vector and is contractible. Therefore, it is no coincidence that the shellings of $IN(M)$ and ${Act}_<(M)$ are related. In fact, we will prove that any shelling order for ${Act}_<(M)$ is a shelling order for $IN(M)$. Theorem \ref{main} then gives:
\begin{equation}\label{shell1}
\textrm{any linear extension of $<_{ext/int}$ is a shelling order for $IN(M)$ and ${Act}_<(M)$.}
\end{equation}

\medskip

%Finally we present the examples
We conclude this section with two examples showing that the linear extensions of the internal and external orders $<_{int}$ and of $<_{ext}$ are not necessarily shelling orders for $Act_<(M)$.

\begin{ex}\label{extfalla}
Consider any linear extension of $<_{ext}$ starting with $124$ and $135$ in that order, such as:
\[
124, 135, 125, 134, 234, 235, 245, 345.
\]
This is not a shelling order for $IN(M)$ because the second facet $135$ intersects the first facet $124$ in codimension 2. By Corollary \ref{ActtoIN} (or directly by inspection), this is not a shelling order for $Act_<(M)$ either. Therefore: 
\begin{equation}\label{shell2}
\textrm{a linear extension of $<_{ext}$ need not be a shelling order for $IN(M)$ or for ${Act}_<(M).$}
\end{equation}
%The poset of the external order for $M$ is shown in Figure \ref{extorder}. Note that the bases $135$ and $124$ are minimal elements, so $F_1=124, F_2=135$ can be extended to a linear order of the poset. For any such extension $F_2\cap F_1$ doesn't have codimension 1, hence they are not shelling orders of $IN(M)$, and by Corollary \ref{ActtoIN}, they are not shelling orders of $Act_<(M)$ either.
\end{ex}

\begin{ex}\label{intfalla}
%The poset of the internal order for $M$ is shown in the left of Figure \ref{intorder}.  
Consider the following linear extension of $<_{int}$:
\[
124, 125, 134, 135, 245, 345, 234, 235
\]
which gives the following order on the facets:
\[
12\overline{2}\overline{4},  \,
12\overline{2}\overline{5}, \,
12\overline{3}\overline{4},  \,
12\overline{3}\overline{5}, \,
2\overline{2}\overline{4}\overline{5},\,
\overline{2}\overline{3}\overline{4}\overline{5}, \,
2\overline{2}\overline{3}\overline{4},  \,
2\overline{2}\overline{3}\overline{5},  \,
\]
%begin{center}
%\begin{tabular}{c}
%Order \\
%$\left\{1,2,3,4,5,\overline{1},\overline{2},\overline{4}\right\}$ \\ 
%$\left\{1,2,3,4,5,\overline{1},\overline{2},\overline{5}\right\}$ \\ 
%$\left\{1,2,3,4,5,\overline{1},\overline{3},\overline{4}\right\}$ \\ 
%$\left\{1,2,3,4,5,\overline{1},\overline{3},\overline{5}\right\}$ \\ 
%$\left\{2,3,4,5,\overline{1},\overline{2},\overline{4},\overline{5}\right\}$  \\ 
%$\left\{3,4,5,\overline{1},\overline{2},\overline{3},\overline{4},\overline{5}\right\}$ \\ 
%$\left\{2,3,4,5,\overline{1},\overline{2},\overline{3},\overline{4}\right\}$  \\ 
%$\left\{2,3,4,5,\overline{1},\overline{2},\overline{3},\overline{5}\right\}$  \\ 
%
%\end{tabular}
%\end{center}
This is a shelling of $IN(M)$ by Theorem \ref{intshell}. However, it is not a shelling of $\textrm{Act}_<(M)$ and ${\textrm{Act}}^\bullet_<(M)$. To see this, suppose we introduce the facets of ${\textrm{Act}}^\bullet_<(M)$ in the order above. When we introduce the sixth facet $\overline{2}\overline{3}\overline{4}\overline{5}$
%of the order, $F_6=\left\{3,4,5,\overline{1},\overline{2},\overline{3},\overline{4},\overline{5}\right\}$,  has
we introduce two new minimal faces: 
%$\left\{\overline{2},\overline{3}\right\}$ and $\left\{\overline{3},\overline{4},\overline{5}\right\}$. 
$\overline{2}\overline{3}$ and $\overline{3}\overline{4}\overline{5}$; so this is 
not a shelling order for $Act_<(M)$. %Note, however, that this is a shelling order for $IN(M)$ as demonstrated in Theorem \ref{intshell}. 
Hence
\begin{equation}\label{shell3}
\textrm{a linear extension of $<_{int}$ is a shelling order for $IN(M)$, but not necessarily for ${Act}_<(M).$}
\end{equation}

\end{ex}

In summary, combining (\ref{shell1}), (\ref{shell2}), and (\ref{shell3}), we see that the hypotheses of Theorems \ref{main} and \ref{intshell} are as strong as possible in the context of LasVergnas's active orders.

%---------------------------------------------------------------------------------Shellabiliity---------------------------------------------------------------
\section{\textsf{Shellability of the external activity complex}}\label{sec:shell}

In this section we prove our main result, which states that the external activity complex is shellable. We begin by proving two technical lemmas. % structural results needed for the proof of shellability of the external activity complex. 

%
%\begin{lem}\label{importante} If $B<_{lex} C$, then there exist $c\in EP(B)\cap C$ and $b<c$ such that $C-c\cup b$ is a basis. 
%\end{lem}
%\begin{proof}  Applying the symmetric exchange property of Lemma \ref{symmexch} to bases $B$ and $C$ and the element $b=\textrm{min}(B-C)$, we obtain an element $c\in C-B$ such that $B-b\cup c\in {\cal B}$ and $C-c\cup b \in {\cal B}$. Since $B <_{lex} C$ we have $b<c$; so $c$ is not the smallest element of $\Circ(B, c)$, and $c\in EP(B) \cap C$. 
%\end{proof}

\begin{lem}\label{masimportante}
Let $M$ be a matroid on an ordered ground set, and let $A, C$ be bases of $M$. There exist $c\in EP(A)\cap C$ and $a<c$ such that $C-c\cup a$ is a basis if and only if $A \ngeq_{ext/int} C$ in LasVergnas's external/internal order.

\end{lem}
\begin{proof}
Given $c \in C$, we can find an element $a<c$ with $C-c\cup a \in \B$ if and only if $c \in IP(C)$. To find such an element $c$ with the additional condition that $c \in EP(A)$, we need $IP(C) \cap EP(A)\neq \emptyset$; this is equivalent to $A \ngeq_{ext/int} C$ in LasVergnas's external/internal order by Theorem \ref{extint}.2.
\end{proof}

A total order $<$ on the set $\B$ of bases of $M$ induces an order on the set of facets $\{F(B)\, : \, B \in \B\}$ of the external activity complex $\textrm{Act}_<(M)$. We now characterize the shelling orders on $\textrm{Act}_<(M)$.

\begin{lem}\label{lema:shelling}
Let $\B$ be the set of bases of a matroid $M$. 
A total order $<$ on the set $\B$ induces a shelling of the external activity complex $\textrm{Act}_<(M)$ if and only if for any bases $A < C$ there exists a basis $B < C$ such that
\begin{enumerate}
\item[(a)] $B = X \cup b$ and $C = X \cup c$ for some $b \neq c$.
\item[(b)] $c \notin A$ and $c \in EA(B)$ if and only if $c \in EA(A)$.
\item[(c)] For any $d \notin B \cup C = X \cup b \cup c$ we have $d \in EA(B)$ if and only if $d \in EA(C)$
\end{enumerate}
\end{lem}

\begin{proof}
By definition, $<$ induces a shelling order if for every $A<C$ there exist $B < C$ and $c^{\pm} \in F(C)$ (where $c^{\pm}$ equals $c$ or $\overline{c}$ for some $c \in E$) such that
\[
F(A) \cap F(C) \subset F(B) \cap F(C) = F(C)-c^\pm.
\]
Recalling that $G(D) = EA(D) \cup \overline{EP(D)}$ is the complement of $F(D)$ in $[[E]]$ for each basis $D$, this is equivalent to 
\[
G(A) \cup G(C) \supset G(B) \cup G(C) = G(C) \cup c^\pm.
\]
Define the support of $S \subset [[E]]$ to be $\supp(S) = \{i \in E \, : \, i \in S \, \textrm{ or } \, \overline{i} \in S\}$. Notice that we have $\supp(G(D)) = E-D$ for any basis $D$. Then
\[
|E|- |B \cap C| = |\supp(G(B) \cup G(C))| = |\supp(G(C) \cup c^\pm)| = |E|-r+1.
 \]
where $r$ is the rank of the matroid. This implies (a).

If (c) was not satisfied for some $d \notin B \cup C$, we would find both $d$ and $\overline{d}$ in $G(B) \cup G(C) = G(C) \cup c$, a contradiction. Finally, 
$c^\pm$ is in $G(A)$ and $G(B)$, which implies (b). 

The converse follows by a very similar argument.
\end{proof}

\begin{cor}\label{ActtoIN}
If a total order $<$ on $\B$ induces a shelling of the external activity complex $\textrm{Act}_<(M)$, then it also induces a shelling of the independence complex $IN(M)$.
\end{cor}

\begin{proof}
Let $A<C$ and assume that $B<C$ satisfy conditions (a), (b), and (c) of Lemma \ref{lema:shelling}. 
Since $\supp(G(D)) = E-D$ for every basis $D$, 
%Now note that $\supp(G(H)) = EA(H) \cup EP(H) = E-H$ for every basis, thus we have that $\supp(G(A)\cup G(C))= E-(A\cap C)$ and $\supp(G(A) \cup G(B)) = E-(A\cap B)$. T
the containment $G(A)\cup G(C) \supset G(B)\cup G(C)$ gives $E-(A\cap C) \supset E-(B\cap C)$, which implies $A\cap C \subset B\cap C = X = C-c$. Hence the total order $<$ induces a shelling order of $IN(M)$. 
\end{proof}

%
%We will show that the lex order on $\cal B$, gives a shelling order of the facets $F(B)$ in $\textrm{Act}_<(M)$. It is enough to prove the following statement:\\
%\\
%Given $B<C$ there exists $A$ with the following properties:
%\begin{itemize}
%\item[{\bf(a)}] $A<C$
%\item[{\bf(b)}] $G(A)\cup G(C) \subset G(B)\cup G(C)$
%\item[{\bf(c)}] $|G(A)\cup G(C)| = n-r+1$
%\end{itemize}
%
%We always have $|G(B)| = n-r$ so property $\bf (c)$ is equivalent to
% \begin{itemize}
%\item[\bf (c$^*$)]There is a unique element $e \in \left( G(A)\cup G(C)\right)-G(C)$.
%\end{itemize}
%Furthermore the set $G(B)$ is supported on $E-B\cup \overline{E-B}$, so $(c^*)$ is in turn is equivalent to: 
%\begin{itemize}
%\item[\bf (c$^{**}$)] $A$ is equal to $C-j\cup i$ for some $i,j \in E$, and every element of $E-(C\cup A)$ belongs to the same external set (passive or active) with respect to both bases.
%\end{itemize}

%------------------------------------------------------PROOF------------------------------------------------------
%\section{Proof that lex ordering is a shelling order:}

Now we are ready to prove our main theorem.

\begin{reptheorem}{main}
Let $M = (E, \mathcal{B})$ be a matroid, and let $<$ be a linear order on the ground set $E$. Any linear extension of LasVergnas's external/internal order $<_{ext/int}$ of $\B$ induces a shelling of the external activity complex $\textrm{Act}_<(M)$.
%
%Let $M$ be a matroid on an ordered ground set.
%%The lexicographic order $<_{lex}$ on the bases 
%Let $<$ be any total order on the bases of $M$ which extends LasVergnas's external/internal order $<_{Ext/Int}$. Then $<$ induces a shelling of the external activity complex.
\end{reptheorem}

\begin{proof}
We use the characterization of Lemma \ref{lema:shelling}. 
Consider bases $A<C$; we will find the desired basis in two steps. We construct a basis $B$ and, if necessary, a second basis $B'$, and we will  show that one of them satisfies the conditions (a),(b),(c) of Lemma \ref{lema:shelling}.

\medskip

\noindent \textbf{Step 1.}
Since $A \ngeq_{ext/int} C$, we first use Lemma \ref{masimportante} to find $c\in EP(A)\cap C$ and a minimal element $b<c$ such that 
\[
B = X \cup b
\]
is a basis, where $X=C-c$. The minimality of $b$ implies that $b$ is minimum in $\Cocirc(B,b)$, so $b \in IA(B)$. Therefore $B \backslash IA(B) \subseteq X \subseteq C$. Theorem \ref{int} then implies that $B <_{int} C$, which in turn gives $B <_{ext/int} C$, and hence $B<C$.
%By Lemma \ref{importante} we have 
%$ with $c\in EP(B)\cap C$ and $a<c$, and $a$ minimal with respect to this property, i.e. if $C-j\cup i'$ is a basis, then $i\leq i'$.

Property (a) is clearly satisfied. By construction $c \notin A$ and $c \in EP(A)$. Since $b<c$ is in $\Circ(B, c)$, we have $c \in EP(B)$. Therefore (b) is also satisfied. Property (c) does not always hold; let us analyze how it can fail, and adjust $B$ accordingly if necessary.

Suppose (c) fails for an element $d \notin B \cup C$; call such an element a \emph{$\{B,C\}$ external disagreement}. This means that $d$ is minimum in one of the fundamental circuits $\beta = \Circ(B,d)$ and $\gamma = \Circ(C,d)$ but not in the other one. 

\begin{figure}[h!]\label{gusanos}
  
  \centering  
    \includegraphics[scale=0.4]{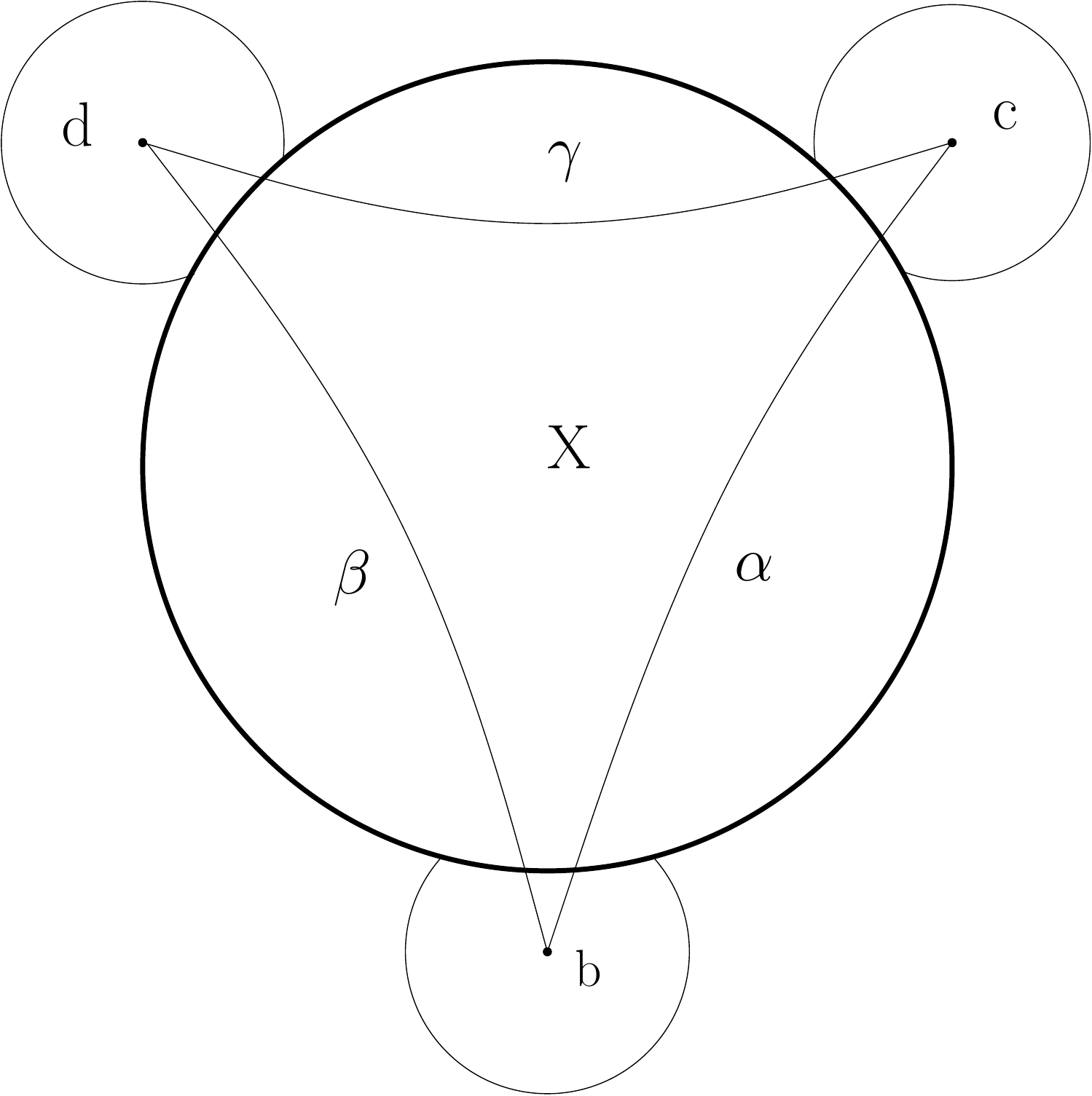}
    \caption{The bases $B = X \cup b$, $C = X \cup c,$ and $D = X \cup d$ and the fundamental circuits $\beta, \gamma, \alpha$.}
   
\end{figure}

Since they have different minima, we have $\beta \neq \gamma$; so using circuit elimination, we can find a circuit $\alpha \subseteq \beta \cup \gamma - d$. This circuit must contain $b$ and $c$, or else it would be contained in basis $B$ or $C$. This implies that
\[
b, c \in \alpha, \qquad b, d \in \beta, \qquad c,d \in \gamma.
\]
It follows that $D = X \cup d = (B \cup d) -b$ is a basis.
By the uniqueness of fundamental circuits, we must have
\[
\alpha = \Circ(B,c) = \Circ(C,b), \quad
\beta = \Circ(B,d) = \Circ(D,b), \quad
\gamma = \Circ(C,d) = \Circ(D,c).
\]

%If property $\bf (c)$, equivalently $\bf (c^{*})$, is satisfied, then the only element in $G(B)\cup G(C)$ which is not in $G(C)$ is $\overline{c}$ (note that $c$ is necessarily passive on $B$ because $b<c$ is on its fundamental circuit) and by construction $\overline{j}\in G(A)$, so $G(B)\cup G(C) \subset G(A)\cup G(C)$, which means property $\bf (b)$ holds too and we are done.\\
%
%If $B$ does not satisfy property $\bf (c)$, equivalently $\bf (c^{**})$, it is because there is some element $d\in E-(B\cup C)$ in different external sets with respect to bases $B$ and $C$. Call such $d$ a \emph{$B$-problematic} element. 

Taking into account that $b<c$, we consider three cases:
\begin{itemize}
\item {\bf 1. $b<c<d$:} Since $b \in \Circ(B,d)= \beta$ and $c \in \Circ(C,d) = \gamma$, $d$ is minimum in neither $\beta$ nor $\gamma$, 
%we have $d \in EP(B)$ and $d \in EP(C)$, 
a contradiction. 
%By lemma \ref{lema1} we have that the fundamental circuits of $d$ with respect to $C$ and $B$ either don't contain $c, b$ or they contain $c$ and $b$ respectively. In the first case both fundamental circuits are the same, so $d$ is passive in both or active in both. In the second case, since $b<c<d$, the element $d$ is externally passive in both bases.
\item {\bf 2. $d<b<c$:} The minimality of $b$ implies that $X \cup d = D$ is not a basis, a contradiction.
%If $c\in \Circ(C,d)$ then $C-c\cup d$ is a basis. Since $d<c$ this cannot happen by the minimality of $b$. By lemma\ref{lema1}, both fundamental circuits must be the same, and $d$ is in the same external set for both bases.
\item {\bf 3. $b<d<c$:} Since $d$ is not minimum in $\Circ(B,d) = \beta \ni b$, we have $d \in EP(B)$; so $d$ is a $\{B,C\}$ external disagreement if and only if 
%if $d$ is $B$-problematic, it must be minimum in $\gamma = \Circ(C,d)$; that is, 
$d \in EA(C)$. 
%
%Again if both fundamental circuits are the same, $d$ is in the same external sets. Otherwise $b\in \Circ(A,d)$, implying that $d$ is externally passive in $B$ and $c \in \Circ(C,d)$.
\end{itemize}

We conclude that, under the above hypotheses, 
\begin{equation}\label{disagree}
\textrm{$d$ is a $\{B,C\}$ external disagreement $\Longleftrightarrow$ $X \cup d$ is a basis, $b<d<c$, and $d \in EA(C)$.} 
\end{equation}
If there are no $\{B,C\}$ external disagreements, $B$ is our desired basis. Otherwise, proceed as follows.

\medskip

\noindent \textbf{Step 2.}
Define the basis
\[
B' = X\cup b'
\]
where $b'$ is the largest $\{B,C\}$ external disagreement. We have $b<b'<c$ and $b' \in EA(C)$. It follows that $B' \subset C \cup EA(C)$, so $B'<_{ext} C$ by Theorem \ref{ext}. This implies that $B'<_{ext/int} C$, which in turn gives $B' < C$. Now we claim that $B'$ satisfies conditions (a),(b),(c) of Lemma \ref{lema:shelling}.

Property (a) is clearly satisfied. By construction $c \notin A$ and $c \in EP(A)$. Since $b'<c$ is in $\Circ(B', c)$, we have $c \in EP(B')$, so (b) holds. To show (c), assume contrariwise that $d' \notin X \cup b' \cup c$ is a $\{B',C\}$ external disagreement; that is, it is minimum in one of the fundamental circuits $\beta'=\Circ(B',d')$ and $\gamma' = \Circ(C, d')$ but not in the other.

 \begin{figure}[h!]\label{gusanos2}  
  \centering  
    \includegraphics[scale=0.4]{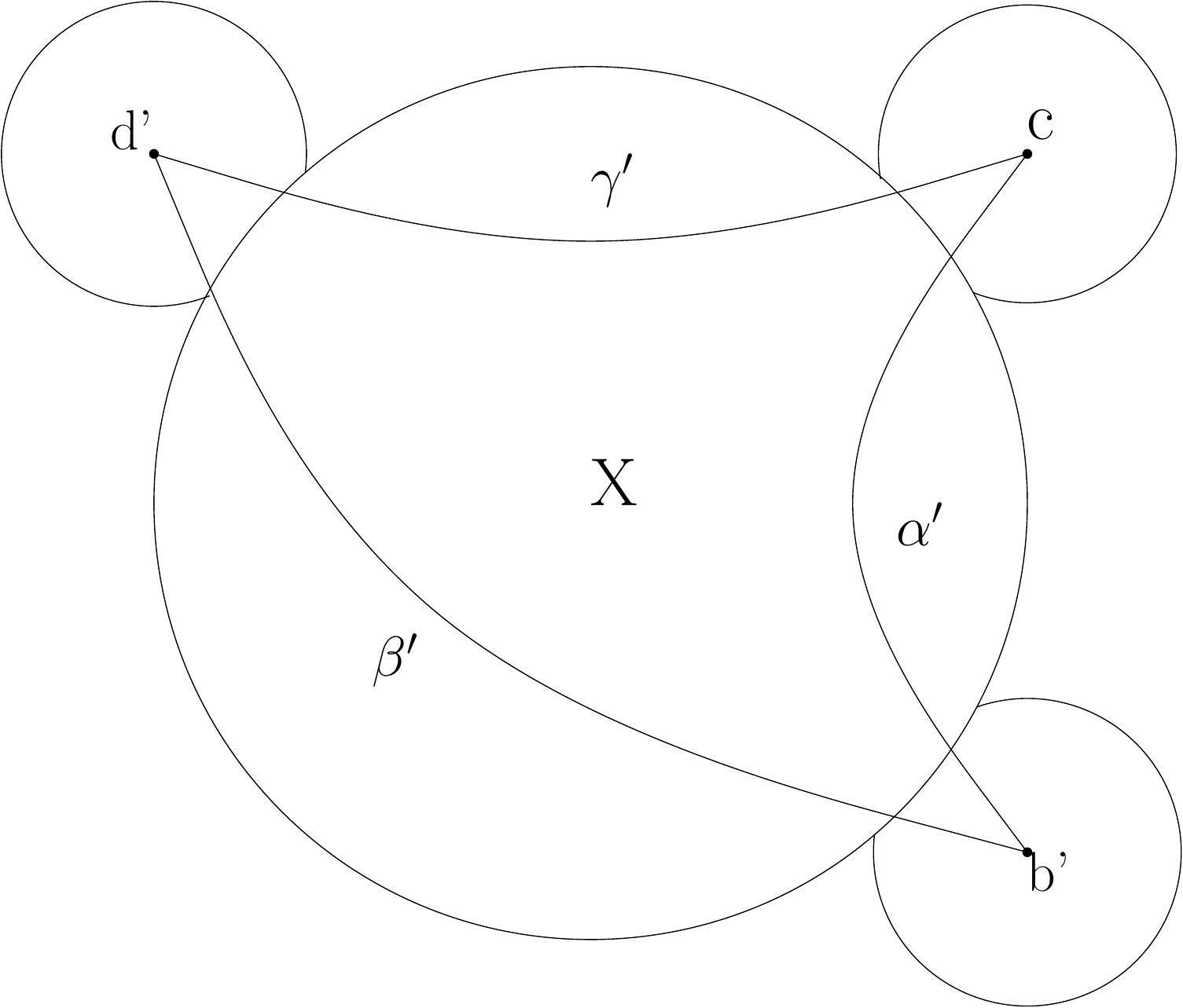}
    \caption{The bases $B' = X \cup b'$, $C = X \cup c,$ and $D' = X \cup d'$ and the fundamental circuits $\beta', \gamma', \alpha'$.}
      \end{figure}

As in Step 1, $D'=X \cup d'$ must be a basis, and we have circuits
\[
\alpha' = \Circ(B',c) = \Circ(C,b'), \quad
\beta' = \Circ(B',d') = \Circ(D',b'), \quad
\gamma' = \Circ(C,d') = \Circ(D',c).
\]
with
\[
b', c \in \alpha', \qquad b', d' \in \beta', \qquad c,d' \in \gamma'.
\]
%Also, by the uniqueness of circuit elimination, we have
%\[
%\alpha' \subseteq \beta' \cup \gamma' - d', \qquad
%\beta' \subseteq \alpha' \cup \gamma' - c, \qquad
%\gamma' \subseteq \alpha' \cup \beta' -  b', 
%\]

Once again, in view of $b'<c$, we consider three cases:
\begin{itemize}
\item {\bf Case 1 $b'<c<d'$:} Since $b' \in \Circ(B,d')= \beta'$ and $c \in \Circ(C,d') = \gamma$, $d'$ is minimum in neither $\beta$ nor $\gamma$, 
%we have $d \in EP(B)$ and $d \in EP(C)$, 
a contradiction. 

\item {\bf Case 2 $d'<b'<c$:} If $d' \in EA(B')$ then $d' = \min \beta'$. Since $b' \in EA(C)$, we have $b' = \min \alpha'$. Because they have different minima, we have $\beta' \neq \alpha'$, so we can use circuit elimination to find a circuit $\gamma'' \subseteq (\alpha' \cup \beta') - b'$. Again, that circuit must contain $c$ and $d'$ or else it would be contained in $C$ or $D'$. Therefore, by the uniqueness of fundamental circuits, $\gamma''=\gamma'$. 
Now, since $\gamma' \subseteq (\alpha' \cup \beta') - b'$ and $d' \in \gamma'$, we have $d' = \min \gamma'$ and $d' \in EA(C)$.

Similarly, if $d' \in EA(C)$ then $d' = \min \gamma'$. Since $b' \in EA(C)$, we have $b' = \min \alpha'$. As above, we can conlcude that $\beta' \subseteq (\alpha' \cup \gamma') - c$ and $d' \in \beta'$, we have $d' = \min \beta'$ and $d' \in EA(B')$.

In either case, we get a contradiction.

%By lemma \ref{lema2}, the element $d'$ has the same activity in $B'$ and $C$.

\item {\bf Case 3 $b'<d'<c$} 
Since $d'$ is not minimum in $\beta' = \Circ(B,d') \ni b'$, if $d'$ is a $\{B', C\}$ external disagreement, it must be minimum in $\gamma = \Circ(C,d')$; that is, $d' \in EA(C)$. We have $b<b'<d'<c$, and $X \cup d'$ is a basis. Therefore, recalling (\ref{disagree}),  $d'$ is also a $\{B, C\}$ external disagreement, contradicting the maximality of $b'$.
%
%
%
%
%Since $b'\in \Circ(B',l)$ so $l$ is externally passive in $B'$. By lemma \ref{lema1} we also have $c\in \Circ(C,l)$. If $l$ is externally active in $C$, then because we also have $b<l<c$ and $c\in \Circ(C,l)$, $l$ would be \emph{$B$-problematic} violating the maximality of $b'$. So $l$ is externally passive in $C$ too, and $l$ is not \emph{$B'$-problematic}.
\end{itemize}

%To check property $\bf (c)$, i.e. that every $l\in E-(B'\cup C)$ has the same external activity in both bases, we consider three cases:
%\begin{itemize}
%\item {\bf Case 1 $b'<c<l$:} Analogous to previous Case 1.
%\item {\bf Case 2 $l<b'<c$:} By lemma \ref{lema2}, the element $l$ has the same activity in both bases.
%\item {\bf Case 3 $b'<l<c$} Again if the fundamental circuits don't contain $b',c$ there is no problem. Otherwise $b'\in \Circ(B',l)$ so $l$ is externally passive in $B'$. By lemma \ref{lema1} we also have $c\in \Circ(C,l)$. If $l$ is externally active in $C$, then because we also have $b<l<c$ and $c\in \Circ(C,l)$, $l$ would be \emph{$B$-problematic} violating the maximality of $b'$. So $l$ is externally passive in $C$ too, and $l$ is not \emph{$B'$-problematic}.
%\end{itemize}

In conclusion, there are no $\{B', C\}$ external disagreements, and property (c) holds. Therefore the basis $B'$ has all required properties.
\end{proof}

%
%\noindent\comment{ We now show that lex can be replaced by any linear extension of the int/ext order (denoted by $\extint$) of Las Vergnas. Let $<$ denote any linear extension of $\extint$. So let $A<C$ and we will construct $B$ in two steps. 
%
%With help of lemma \ref{masimportante} find $c\in EP(A) \cap C$ and minimal element $b<c$ such that $X\cup b$ is basis. Note that $B$ is the smallest lexicographic basis containing $X$, so by Las Vergnas prop 5.2(iv), we get that $B<_{\text{int}} C$, thus $B\extint C$ and, in particular, $B<C$. Again, properties (a) and (b) are satisfied and the one that can fail is (c). 
%
%Suppose that $c$ fails for some $d\notin B\cup D$; call it $B$-problematic. Exactly as before, we conclude that $d$ is $B$-problematic if and only if $X\cup d$,  $b<d<c$ and $d\in EA(C)$. So let $B' = X\cup b'$, where $b'$ is the largest $B$-problematic element. 
%
%As in the previous part, we claim that conditions $(b)$ and $(c)$ are verified. The 'tricky' part is to show that $B'< C$. However, note that $b'<c$ and $b'\in EA(C)$. It follows that $F(C)\cap F(B') = F(C)-\overline c = F(B')-b'$. Intersecting this equation with $\overline{[n]}$ we get that $C\cup EA(C)-c = B'\cup EA(B')$ and in particular $B'\subseteq C\cup EA(C)$. As a result $B'<_{\text{ext}}C$, from what we get $B'\extint C$ and $B'<C$ as desired. 
%}
%
%

%\comment{Pensaba que las extensiones de $<_{int}$ tambiŽn dan shelling, pero ahora creo que no. En todo caso, deber'amos comentar que esto tambiŽn da cosas sobre los shellings de la matroide. Eso yo creo que es importante!}

\begin{cor}{}
Any linear extension of the external/internal order $<_{ext/int}$ gives a shelling order for the independence complex $IN(M)$. 
\end{cor}

\begin{proof}
This follows from Theorem \ref{main} and Corollary \ref{ActtoIN}. 
\end{proof}

In fact, we now prove a stronger result. We begin with a useful lemma.

%In general, the independence complex $IN(M)$ has shelling orders that are not shelling orders of $Act_<(M)$. The following theorem, together with the example from Section 6, provides an example of such an order.
%
%\begin{proof}
%Let $<$ be a linear extension of $<_{int}$ and consider bases $A < C$. Then $A\not >_{int} C$. We claim that there exists $B<_{int} C$ (and hence $B < C$) such that $A\cap C \subseteq B\cap C = C-c$ for some $c$ in $C$; this will prove the desired result.
%
%Since $A\not >_{int} C$, $C$ is not the lexicographically smallest basis that contains $A\cap C$. Let $D$ be that basis, so $D <_{lex} C$. 
%%Then $D<_{int} A$ and $D<_{int} C$, and $D\not=C$. Since lex is a linear extension of $<_{int}$, we have that $D$ is smaller in lex than $C$. 
%Let $d$ be the smallest element in $D-C$ and choose any $c \in \Circ(C,d)$, so $C-c \cup d$ is a basis. We have $c \geq \min(C-D) > \min(D-C) = d$. 
%
%Put $X =C-c$ and let $B = X \cup b$ be the lexicographically smallest basis that contains $X$. Since $D<_{lex} C$ is one such basis, $B \not= C$. Also, $B<_{int} C$ since $B$ is the lexicographically smallest basis that contains $X = B \cap C$. Finally, since $A \cap C \subseteq D$, we have $A \cap C \subseteq D \cap C = X$, so $A\cap C\subseteq B\cap C = X=C-c$. Therefore $B$ satisfies the desired properties.
%\end{proof}
%}
%We start with a useful lemma.

\begin{lem}\label{minimo}
Let $I$ be an independent set of $M$ and let $C$ be any basis that contains $I$. If $B$ is the lexicographically smallest basis that contains $I$ then $B\leq_{int} C$.
\end{lem}
\begin{proof} 
By Theorem \ref{int}.4, we need to show that $B$ is the lexicographically smallest basis that contains $B\cap C$. To do so, assume there is a basis $A \supseteq B\cap C$ with  $A<_{lex} B$ Then $A \supseteq B\cap C\supseteq I$, % then $I\subseteq A$ 
contradicting the minimality of $B$. 
\end{proof}

\begin{reptheorem}{intshell} Any linear extension of the internal order $<_{int}$ gives a shelling order of the independence complex $IN(M)$.
\end{reptheorem}

\begin{proof}
Let $<$ be any linear extension of $<_{int}$, and let $A<C$ be bases, so 
%Let 
$A\not >_{int} C$. %be bases. 
We claim that there exists $B<_{int} C$ (and hence $B<C$) such that $A\cap C \subseteq B\cap C = C-c$ for some $c$ in $C$. This will prove the desired result.

To show this, let $D$ be the lexicographically smallest basis that contains $A\cap C$.  % \red{POR QUE?}\green{Ahora si?}
Notice that $D\not=C$ because $A\not >_{int} C$, using Theorem \ref{int}.4. Let $d$ be smallest element in $D-C$ and let $c$ be any element of $C-D$ such that $C'=C-c \cup d$ is a basis. Also 
%Then $D<_{int} A$ and 
notice that $D<_{int} C$  by Lemma \ref{minimo}; and since $<_{lex}$ is a linear extension of $<_{int}$, we have $D <_{lex} C$. This gives $d = \min(D-C) < \min(C-D) \leq c$, and therefore $C' <_{lex} C$.

Put $X =C-c$ and let $B$ be the lexicographically smallest basis that contains $X$. Since $C'$ contains $X$, $B \leq_{lex} C' <_{lex} C$, so $B \neq C$. Therefore
%
%Since $C' \supset X$, Lemma \ref{minimo} implies that $B \le_{lex} C'$, and therefore $B \le_{lex} C$. 
%
%
%We have that $C'$ contains $X$
%%We have that $D$ contains $X$ \red{CREO QUE NO}\green{Estaba mal explicado, pero creo que ya esta claro.}, 
%so $B \le_{lex} D <_{lex} C$. 
$B<_{int} C$ by Lemma \ref{minimo}. 
Also note that, since $c \notin D \supset A \cap C$ and $c \in C$, we must have $c \notin A$.  This gives $A\cap C\subseteq C-c=X$, and therefore $A\cap C\subseteq B\cap C = X$. It follows that $B$ satisfies the desired properties.
\end{proof}

%\begin{rem} 
%To see that the hypotheses of Theorem \ref{main} and Theorem \ref{intshell} are at the correct level of generality with respect to L%asVergnas's active orders, we give examples in Section 6 showing that linear extensions of $<_{int}$ and $<_{ext}$ may not be shelling orders of $Act_<(M)$, and linear extensions of $<_{ext}$ may not be shelling orders of $IN(M)$.
%However, see Proposition \ref{intshell}.
%\end{rem}

%------------------------------------------------------H-VECTOR------------------------------------------------------
\section{\textsf{The $h$-vector}}\label{sec:h}

We now describe the restriction sets for the shellings of Theorem \ref{main}.

\begin{prop}\label{lemR}
Let $<$ be any linear extension of $<_{ext/int}$, and regard it as a shelling order for $IN(M)$. Then the restriction set of each facet $C$ (which is a basis of $M$) is $IP(C)$.
\end{prop}

\begin{proof} We need to show $IP(C)$ is the minimum subset of $C$ which is not a subset of a basis $B<C$. 

To show that $IP(C)$ indeed has this property, assume that if $IP(C) \subseteq B$. Then by Theorem \ref{int}.2, we have $C \leq_{int} B$ and hence $C \leq B$, as desired.

To show minimality, let $U \subsetneq IP(C)$. By Theorem \ref{crapo}.2 we can find a basis $A$ such that $A - IA(A) \subseteq U \subseteq A$. This gives $A-IA(A) \subseteq U \subsetneq C - IA(C)$, which in light of Theorem \ref{int}.3 gives $A <_{int} C$, and hence $A < C$. Therefore $U$ is a subset of $A$ with $A<C$, as desired.
\end{proof}

\begin{prop} 
Let $<$ be any linear extension of $<_{ext/int}$, and regard it as a shelling order for $\textrm{Act}_<(M)$.
Then the restriction set of each facet $F(C)$  (where $C$ is a basis of $M$) is $\overline{IP(C)}$.
% and $IN(M)$. 
%For every basis $C$, the restriction sets of the corresponding bases in these shellings satisfy ${\cal R}(F(C), \textrm{Act}_<(M)) = \overline{{\cal R}(C , IN(M))}$.
\end{prop}

\begin{proof} 
We need to show $\overline{IP(C)}$ is the minimum subset of $F(C)$ which is not a subset of $F(B)$ for any basis $B<C$.
%Let $\RR(C) := {\cal R}(C , IN(M))$ and $\RR(F(C)) = {\cal R}(F(C), \textrm{Act}_<(M))$. 
%First we show that ${\cal R}(C) = IP(C)$. This is known for the lexicographic order $<_{lex}$ \cite[\comment{cita exacta?}]{bjorner} 

To show $\overline{IP(C)}$ does have this property, assume that $\overline{IP(C)} \subseteq F(B) = B \cup EP(B) \cup \overline{B \cup EA(B)}$ for some basis $B$. Then $IP(C) \subset B \cup EA(B)$, so $IP(C) \cap EP(B) = \emptyset$. By Theorem \ref{extint}.2, $C <_{ext/int} B$ so $C<B$, as desired.

To show minimality, let $\overline{U} \subsetneq \overline{IP(C)}$, so $U \subsetneq IP(C)$. By Proposition \ref{lemR}, $U$ is contained in a basis $A<C$, and hence $\overline{U}$ is contained in $F(A)$ for that basis, as desired.
\end{proof}

As an immediate consequence, we obtain our main enumerative result.

\begin{reptheorem}{hvector} 
The $h$-vector of $\textrm{Act}_<(M)$ equals the $h$-vector of $M$.
%The $h$-vector of $M$ is equal to the $h$-vector of $\textrm{Act}_<(M)$.
\end{reptheorem}
\begin{proof}
This follows from the previous two results, in light of Proposition \ref{shellingh}.
%Recall the definition
%\[h_i = \textrm{card}\left\{F\textrm{ a facet} \, : \,  |{\cal R}(F)| = i\right\}\]
%and by the previous theorem the restriction sets have the same size for $B$ and $F(B)$.
\end{proof}

%\noindent\comment{We can also prove the same statement about the restriction sets! They are all equal, and equal to the ones coming from lex.}

\section{\textsf{Topology}}\label{sec:topology}

The external activity complex $\textrm{Act}_<(M)$ is a cone; for example, it is easy to see that every facet contains $\overline{\min E}$ and $\max E$. Therefore $\textrm{Act}_<(M)$ is trivially contractible. It is more interesting to study the topology of the \emph{reduced external activity complex} ${\textrm{Act}}^\bullet_<(M)$, obtained by removing all cone points of $\textrm{Act}_<(M)$. It turns out that Corollary \ref{hvector} gives us enough information to describe it.
First we need a few technical lemmas. 

\begin{defi}
Define a \emph{loop} of a simplicial complex $\Delta$ to be an element $l$ of the ground set such that $\{l \}$ is not a face of $\Delta$.
\end{defi}
\begin{defi}
An element $e$ of a matroid $M$ is \emph{absolutely externally active} if it is externally active with respect to every basis not containing it, or \emph{absolutely externally passive} if it is externally passive with respect to every basis not containing it.
\end{defi} 
Let $AEA(M)$ and $AEP(M)$ be the respective sets of elements, and call the elements of $AE(M) = AEA(M) \cup AEP(M)$  \emph{externally absolute}. 

\begin{lem} \label{groundset}
The set of cone points of ${Act}_<(M)$ is $AEP(M) \cup \overline{AEA(M)}$. 
%
%those $e$ which are absolutely  \, : \, e \notin AEP(M)\} \cup \{\overline{e} \, : \, e \notin AEA(M)\}$. 
The ground set of ${\textrm{Act}}^\bullet_<(M)$ is $\{e \, : \, e \notin AEP(M)\} \cup \{\overline{e} \, : \, e \notin AEA(M)\}$, and this simplicial complex has no loops.
\end{lem}

\begin{proof}
The first two statements are clear from the definitions. For the last one, if $e \notin AEP(M)$, then we can find a basis $B$ with respect to which $e$ is externally active, so $\{\overline{e}\}\subset F(B)$ is a face of ${\textrm{Act}}^\bullet_<(M)$. Similarly, if $e \notin AEA(M)$, then we can find a basis $B$ with respect to which $e$ is externally passive, so $\{e\} \subset F(B)$ is a face of ${\textrm{Act}}^\bullet_<(M)$. 
\end{proof}
%
% Notice that $j \in S$ (resp. $\overline{j} \in S)$  if and only if $j$ is universally externally passive (resp. universally externally active). 

\begin{lem}\label{circuitos} 
Let $M=(E,{\cal B})$ be a matroid. Every element $e\in E$ is externally absolute if and only if  
%passive for every basis that does not contain it or externally active for every basis that does not contain it, 
the circuits of $M$ are pairwise disjoint.
\end{lem}
\begin{proof}
The backward direction is a straightforward consequence of the definitions.
To prove the forward direction, we proceed by contradiction. % Call an element with that property an \emph{absolute} element. We will show that if two circuits intersect, then there is some element which is not absolute.\\\\
Assume that every element of $M$ is externally absolute, and that we have two circuits $\gamma_1$ and $\gamma_2$ with $\gamma_1\cap \gamma_2\neq \emptyset$ whose minimal elements are $c_1$ and $c_2$, respectively. Consider two cases.

1. If $c_1=c_2$ then perform circuit elimination to get $\gamma_3\subset \gamma_1\cup \gamma_2 - c_1$. Let $c_3$ be the minimal element of $\gamma_3$; without loss of generality assume $c_3 \in \gamma_1$. Then $c_3$ is externally active for some basis, as testified by $\gamma_3$, and it is  externally passive for another basis, as testified by $\gamma_1$. Hence $c_3$ is not absolute, a contradiction

2. If $c_1\neq c_2$ and $c\in \gamma_1\cap \gamma_2$, then perform circuit elimination with $c$ to get a circuit $\gamma_3 \subset \gamma_1\cup \gamma_2 - c$. Let $c_3$ be the minimal element of $\gamma_3$; assume $c_3 \in \gamma_1$. If $c_3=c_1$, then case 1 applies to circuits $\gamma_1$ and $\gamma_3$, and we get a contradiction. 
Otherwise, we must have $c_1 < c_3$ since $c_1 = \min \gamma_1$. Therefore $c_3$ is externally active for some basis, as testified by $\gamma_3$, and externally passive for another basis, as testified by  $\gamma_1$, a contradiction.
\end{proof}

\begin{prop}\label{subcomplex}
If a matroid is the disjoint union of circuits, then ${\textrm{Act}}^\bullet_<(M) \cong IN(M)$. 
Otherwise, ${\textrm{Act}}^\bullet_<(M)$ has a proper subcomplex which is isomorphic to $IN(M)$. 
The embedding may be chosen so that the image of facet $B$ of $IN(M)$ is a subset of the facet $F(B)$ of ${\textrm{Act}}^\bullet_<(M)$.
%\comment{No vale la pena decir que cada base $B$ vive dentro de $F(B)$ y que las igualdades de los shellings cuadran bien dentro del embedding? }
%
%The reduced external activity complex $\underline{\textrm{Act}}_<(M)$ has a subcomplex which is isomorphic to the independence complex $IN(M)$. 
\end{prop}

\begin{proof}
For every $e \in E$ let $e' = e$ if $e$ is absolutely externally active, and $e' = \overline{e}$ otherwise. The set $E' = \{e' \, : \, e \in E\}$ is a subset of the vertices of ${\textrm{Act}}^\bullet_<(M)$ by Lemma \ref{groundset}. For every basis $B$ of $M$ the set $B'=\{b' \, : \,  b \in B\}$ is a subset of $F(B)$, and hence a face of ${\textrm{Act}}^\bullet_<(M)$. This gives the desired embedding of $IN(M)$ in ${\textrm{Act}}^\bullet_<(M)$.

If $M$ is the disjoint union of circuits, then $E'$ equals the ground set of ${\textrm{Act}}^\bullet_<(M)$, and $B'$ equals $F(B) \cap E'$ for all bases $B$, so this embedding is actually an isomorphism. 

If $M$ is not the disjoint union of circuits, by Lemma \ref{groundset}, $E'$ is a proper subset of the ground set of ${\textrm{Act}}^\bullet_<(M)$, so the embedding of $IN(M)$ is a proper subcomplex of ${\textrm{Act}}^\bullet_<(M)$.
\end{proof}

\begin{lem}\label{sphere}
If a matroid $M$ of rank $r$ is the disjoint union of circuits, then the independence complex $IN(M)$ is homeomorphic to an $(r-1)$-sphere.
\end{lem}

\begin{proof} 
If $M$ is a single circuit (necessarily of size $r+1$), then $IN(M)$ is the boundary of an $r$-simplex, and hence an $(r-1)$-sphere. %We then get $h_r(IN(M)) = 1$.

If $M$ is the disjoint union of circuits $\gamma_1, \ldots, \gamma_k$ then $IN(M)$ is the join of $IN(\gamma_1), \ldots, IN(\gamma_k)$; that is, $IN(M) = IN(\gamma_1) \star \cdots \star IN(\gamma_k) = \{A_1 \cup \cdots \cup A_k \, : \, A_i \in IN(\gamma_i) \textrm{ for }1 \leq i \leq k\}$. The result then follows from the fact that the join of two spheres $\mathbb{S}^k$ and $\mathbb{S}^l$ is homeomorphic to the sphere $\mathbb{S}^{k+l+1}$. \cite[Chapter 2.2.2]{Kozlov}
% From the description of the $h$-polynomial in terms of the $f$-vector it then follows that $h_M(x) = h_{C_1}(x) \cdot \cdots \cdot  h_{C_k}(x)$. Since the top entry of the $h$-vector of each $C_i$ is 1 by (a), we obtain $h_r(IN(M)) = 1$ as well.
\end{proof}

The matroids with pairwise disjoint cycles have a nice characterization in terms of excluded minors. 

\begin{lem} \label{U31}  A matroid $M$ contains two circuits with non empty intersection if and only if $U_{3,1}$ is a minor of $M$.
\end{lem}

\begin{proof} First suppose that $M$ contains two intersecting circuits $\gamma$ and $\delta$ which intersect at $e$. Let $c \in \gamma - \delta$ and $d \in \delta - \gamma$. Restricting to $\gamma \cup \delta$ and then contracting every element except for $c,d,$ and $e$, we obtain $U_{3,1}$ as a minor.

To show the converse consider any matroid $N$ and an element $e\in E$. Notice that every circuit of $N \backslash e$ is a circuit of $N$; and if $\gamma$ is a circuit of $N$, then either $\gamma$ or $\gamma \cup e$ is a circuit of $N$. It follows that if either $N \backslash e$ or $N/e$ have two overlapping circuits, so does $N$. Since $U_{3,1}$ has two overlapping circuits, so does every matroid containing it as a minor.
%
%
% If $e$ is not a coloop, all the circuits of $M\backslash e$ are also circuits in $M$. Also, the if $C$ is a circuit of $M\slash e$, then $C\cup \{e\}$ is a circuit in $M$. It follows if $M'$ is a minor of $M$ and $C'$ is a circuit of $M'$, then there is a circuit $C$ of $M$ such that $C\cap E(M') = C'$. 
%
%Assume that $U_{3,1}$ is a minor of $M$ and let $\gamma_1, \gamma_2$ be two distinct circuits of the minor $M' \cong U_{3,1}$. Then there are two circuits $C_1, C_2$ of $M$ with $\gamma_1 = C_1\cap E(M')$ and $\gamma_2 = C_2 \cap E(M')$. 
\end{proof}
Now we are ready to prove our main topological result.

\begin{reptheorem}{topo} %Assume that $M$ is coloop free. 
Let $M$ be a matroid and $<$ be a linear order on its ground set. The reduced external activity complex ${\textrm{Act}}^\bullet_<(M)$ is contractible if $M$ contains $U_{3,1}$ as a minor, and a sphere otherwise.
\end{reptheorem}

\begin{proof} Notice that if $M$ has a coloop $c$, then both $c$ and $\overline{c}$ are cone points of $\textrm{Act}_<(M)$, and are invisible in ${\textrm{Act}}^\bullet_<(M)$. Therefore we may assume that $M$ is coloop free. 

Let $r$ be the rank of $M$, and let $d = \dim({\textrm{Act}}^\bullet_<(M)) = \dim({Act}_<(M)) - |AE(M)| = n+r-1-|AE(M)|$. We consider two cases.

\medskip

1. If $M$ is not the disjoint union of circuits,  $|AE(M)| < n$ by Lemma \ref{circuitos}, so 
$d > r-1$. Clearly $h_d({\textrm{Act}}^\bullet_<(M)) = h_d({Act}_<(M))$, Theorem \ref{hvector} gives $h_d({Act}_<(M)) = h_d(IN(M))$, and since $IN(M)$ is $(r-1)$-dimensional, $h_d(IN(M)) = 0$.  
Therefore, by
%$h_d(\widetilde{\textrm{Act}})=0$ and 
%Since $\underline{\textrm{Act}}_<(M)$ is shellable, 
Theorem \ref{contraible}, ${\textrm{Act}}^\bullet_<(M)$ is contractible.

\medskip

2. If $M$ is the disjoint union of circuits, then ${\textrm{Act}}^\bullet_<(M) \cong IN(M)$ is a sphere invoking Proposition \ref{subcomplex} and Lemma \ref{sphere}.

\medskip

The result follows from Lemma \ref{U31}.
\end{proof}

We conclude that the simplicial complex ${\textrm{Act}}^\bullet_<(M)$ is a model for a matroid $M$ which is topologically simpler than the ``usual" model $IN(M)$.

\section{\textsf{Questions}}

\begin{itemize}
\item
There should be ``affine" analogs of the results of this paper. Geometrically, they should correspond to taking the closure of an affine subspace $L$ of $\mathbb{A}^n$ in $(\mathbb{P}^1)^n$, as opposed to a linear subspace, as explained in \cite{ArdilaBoocher}. To a \emph{morphism of matroids} $M \rightarrow M'$, one may associate an external activity complex $\textrm{Act}_<(M \rightarrow M')$ \cite{ArdilaBoocher} and active orders $<_{int}, <_{ext}, <_{ext/int}$ \cite{LasVergnas}. The analogous foundational results, such as Theorems \ref{crapo}, \ref{ext}, \ref{int}, \ref{extint} hold there as well. \cite{semimatroids, LasVergnas}
Do our main theorems hold in that more general setting?
\item
Even though $\textrm{Act}_<(M)$ only pays attention to the external activities of the bases of $M$, it is the external/internal order $<_{ext/int}$ which plays a crucial role in its shelling. This makes the following question from \cite{ArdilaBoocher} even more natural: is $\textrm{Act}_<(M)$ part of a larger (and well-behaved)  simplicial complex which simultaneously involves the internal and external activities of the bases of M? Ideally we would like it to come from a natural geometric construction.
\item Notice that for an ordered matroid $M$, every linear extension of the poset of restriction sets of the lexicographic shelling order of $IN(M)$ gives another shelling order with the same restriction sets. That means that every posible order of the facets that could give a shelling with the same restriction sets gives another shelling of $IN(M)$. Does this property say something more about the independence complex. Is there a wide class of examples of a shellable complex with a fix shelling order, such that every linear extension of the poset is again a shelling. Notice that \ref{intfalla} is an example that $Act_<(M)$ with the associated lexicographic shelling does not have this property. 
\end{itemize}

\section{\textsf{Acknowledgments}}

The first author would like to thank Adam Boocher; this project would not exist without our collaboration in \cite{ArdilaBoocher}, and the numerous conversations with him on this topic. The second author would like to Isabella Novik for the invitation to the University of Washington, where part of this project was carried on. This project started at the Encuentro Colombiano de Combinatoria (ECCO 14) and we would like to thank the SFSU-Colombia Combinatorics Initiative and the Universidad de Los Andes for hosting and funding the event.

\bibliography{biblio}
\bibliographystyle{plain}
\end{document}